\documentclass[11pt]{amsart}
\usepackage{amsmath}
\usepackage{latexsym,amsfonts,amssymb,mathrsfs}
\usepackage{mathdots}
\usepackage{color}
\usepackage[all,2cell,color]{xy}
\usepackage{enumitem} 
\usepackage{bbm}
\usepackage{tikz}
\usepackage{tikz-cd}
\usepackage[T1]{fontenc}
 \usepackage[utf8]{inputenc}
 \usepackage{blkarray} 
 \usepackage{ifthen}
 \usepackage{hyperref}
\usepackage[capitalise]{cleveref}

\UseAllTwocells
\usepackage[
textwidth=3cm,
textsize=small,
colorinlistoftodos]
{todonotes}

\usetikzlibrary{positioning} 
\usetikzlibrary{decorations.pathreplacing}
\usetikzlibrary{arrows, decorations.markings} 
\usetikzlibrary{decorations.pathmorphing}
\pgfarrowsdeclarecombine{twotriang}{twotriang}{stealth}{stealth}%
{stealth}{stealth}
\tikzset{arrow/.style={-stealth}}
\tikzset{arrowshorter/.style={-stealth, shorten <=2pt, shorten >=2pt}}
\tikzset{arrowmuchshorter/.style={-stealth, shorten <=7pt, shorten >=6pt}}
\tikzset{mono/.style={>-stealth}} 
\tikzset{epi/.style={-twotriang}} 
\tikzset{twoarrowlonger/.style={double,double distance=1.5pt,
shorten <=5pt,shorten >=6pt,
decoration={markings,mark=at position -4pt with {\arrow[scale=1.75]{>}}},
preaction={decorate}}} 

\tikzset{twoarrow/.style={double,double distance=1.5pt,
shorten <=6pt,shorten >=7pt, 
decoration={markings,mark=at position -4pt
with {\arrow[scale=1.75]{>}}},
preaction={decorate} 
}
}
\tikzset{%
    symbol/.style={%
        draw=none,
        every to/.append style={%
            edge node={node [sloped, allow upside down, auto=false]{$#1$}}}
    }
}

\tikzset{mapstikz/.style={-stealth, 
decoration={markings,mark=at position 0pt with {\arrow[scale=0.5]{|}}}, preaction={decorate}}}
\usetikzlibrary{backgrounds,shapes,arrows,calc,patterns} 

\pgfdeclarelayer{background}
\pgfsetlayers{background,main}

\theoremstyle{plain}   
\newtheorem{thm}{Theorem}[section] 
\makeatletter\let\c@thm\c@thm\makeatother
\newtheorem{cor}{Corollary}[section]
\makeatletter\let\c@cor\c@thm\makeatother
\newtheorem{lem}{Lemma}[section]
\makeatletter\let\c@lem\c@thm\makeatother
\newtheorem{prop}{Proposition}[section]
\makeatletter\let\c@prop\c@thm\makeatother

\makeatletter\let\c@claim\c@thm\makeatother

\makeatletter\let\c@conjecture\c@thm\makeatother

\newtheorem*{unnumberedtheorem}{Main Theorem}

\theoremstyle{definition}

\newtheorem{defn}{Definition}[section]
\makeatletter\let\c@defn\c@thm\makeatother

\makeatletter\let\c@const\c@thm\makeatother
\newtheorem{notn}{Notation}[section]
\makeatletter\let\c@notn\c@thm\makeatother

\theoremstyle{remark}

\newtheorem{rmk}{Remark}[section]
\makeatletter\let\c@rmk\c@thm\makeatother
\newtheorem{ex}{Example}[section]
\makeatletter\let\c@ex\c@thm\makeatother

\makeatletter\let\c@observation\c@thm\makeatother

\makeatletter\let\c@warning\c@thm\makeatother

\makeatletter\let\c@digression\c@thm\makeatother

\makeatletter\let\c@answ\c@thm\makeatother

\makeatletter
\let\c@equation\c@thm
\numberwithin{equation}{section}
\makeatother

\newcommand{\newrefformat}[2]{}

\crefname{lem}{Lemma}{Lemmas}
\crefname{thm}{Theorem}{Theorems}
\crefname{defn}{Definition}{Definitions}
\crefname{notn}{Notation}{Notations}
\crefname{const}{Construction}{Constructions}
\crefname{prop}{Proposition}{Propositions}
\crefname{rmk}{Remark}{Remarks}
\crefname{cor}{Corollary}{Corollaries}
\crefname{equation}{Display}{Displays}
\crefname{ex}{Example}{Examples}
\crefname{thmalph}{Theorem}{Theorems}
\crefname{answ}{Answer}{Answers}

\newcommand{\cM}{\mathcal{M}}

\newcommand{\cS}{\mathcal{S}}

\newcommand{\set}{\cS\!\mathit{et}}

\newcommand{\msset}{m\mathit{s}\set}

\DeclareMathOperator{\id}{id}

\newcommand{\aamalg}[1]{\underset{#1}{{\amalg}}} 
\DeclareMathOperator{\op}{op}

\DeclareMathOperator{\pr}{pr}

\newcommand{\eqDelta}{\Delta[3]_{eq}}

\newcommand{\DeltaThree}[3]{\ifthenelse{\equal{#1}{}}{\ifthenelse{\equal{#2}{}}{\Delta[3_{#3}]}{\Delta[3_{#3}|#2]}}{\ifthenelse{\equal{#2}{}}{\Delta[#1|3_{#3}]}{\Delta[#1|3_{#3}|#2]}}}

\title{Gray tensor product and saturated $N$-complicial sets}

\author{Viktoriya Ozornova}
\address{Fakult\"at f\"ur Mathematik, Ruhr-Universit\"at Bochum, 44780 Bochum, Germany}
\email{viktoriya.ozornova@rub.de}

\author{Martina Rovelli}
\address{Mathematical Sciences Institute, Australian National University,
ACT 2601, Australia
}
\email{martina.rovelli@anu.edu.au}

\author{Dominic Verity}
\address{Centre of Australian Category Theory, Macquarie University, NSW 2109,
Australia
}
\email{dominic.verity@mq.edu.au}


\begin{document}

\maketitle

\begin{abstract}
    We show that the pretensor and tensor products of simplicial sets with marking are compatible with the homotopy theory of saturated $N$-complicial sets (which are a proposed model of $(\infty,N)$-categories), in the form of a Quillen bifunctor and a homotopical bifunctor, respectively.
\end{abstract}

\addtocontents{toc}{\protect\setcounter{tocdepth}{0}}

\section*{Overview}
Higher category theory is becoming increasingly important as a unifying language for various areas of mathematics, most notably for algebraic topology and algebraic geometry, where many relevant structures occur naturally as \emph{$(\infty,N)$-categories}, rather than strict $N$-categories. In this article, we are concerned with an $(\infty,N)$-categorical version of the \emph{Crans--Gray tensor product} \cite{GrayFormalCategoryTheory,CransOmega}, originally defined for strict $N$-categories in order to encode different flavors of lax natural transformations.

The model of $(\infty,N)$-categories that we consider, due to
the third-named author,
is that of \emph{saturated $N$-complicial sets}. A saturated $N$-complicial set is a simplicial set with marking satisfying extra conditions that guarantee that the marked simplices behave as higher equivalences.
In \cite{VerityComplicialAMS}, he constructed two pointset models of the Gray tensor product of simplicial sets with marking: the tensor $\otimes$ and the pretensor $\boxtimes$, homotopically equivalent but each with different valuable properties,
and showed that they are compatible with the homotopy theory of (non-saturated) $N$-complicial sets.

In this note, we provide the extra verification that enables us to conclude that the pretensor and the tensor products $\boxtimes$ and $\otimes$ are in fact also compatible with the model structure for \emph{saturated} $N$-complicial sets, in a sense that will be made precise by \cref{tensorhomotopical,pretensorbiquillen}.

\begin{unnumberedtheorem}
 For any $N\in\mathbb N$, the bifunctors $\boxtimes$ and $\otimes$ are homotopical with respect to the model structure on simplicial sets with marking for saturated $N$-complicial sets, which model $(\infty,N)$-categories.
\end{unnumberedtheorem}

The theorem was proven for $N=1$ by Joyal \cite[Thm 6.1]{JoyalVolumeII} in the context of quasi-categories and by Lurie \cite[Cor.~3.1.4.3]{htt} in the context of marked simplicial sets. During the final work on the completion of this paper, analogous result was shown for $N=2$ by Gagna--Harpaz--Lanari \cite{GagnaHarpazLanariGray} in the context of scaled simplicial sets.
For general $N$,
the result was previously obtained by the third-named author, and recently rediscovered by the first two authors.

Beside for its own interest, the result would play a role in work by Campion--Kapulkin--Maehara, in comparing cubical models of $(\infty,N)$-ca\-te\-go\-ries to saturated $N$-complicial sets, as indicated in \cite[Rmk~7.3, Conj.~7.4]{CampionKapulkinMaehara}.

\addtocontents{toc}{\protect\setcounter{tocdepth}{1}}
\subsection*{Acknowledgements}
  We would like to thank Emily Riehl for bringing the problem treated in this paper to the attention of the first two authors, and Len\-nart Meier for helpful conversations on this project.
{This material is based upon work supported by the National Science Foundation under Grant No.\ DMS-1440140 while the authors were in residence at the Mathematical Sciences Research Institute in Berkeley, California, during the Spring 2020 semester. The first-named author thankfully acknowledges the financial support by the DFG grant OZ 91/2-1 with the project nr.~442418934. The third-named author was supported by the Discovery Project DP190102432 from the Australian Research Council.}
\tableofcontents
\addtocontents{toc}{\protect\setcounter{tocdepth}{2}}
\section{Background on simplicial sets with marking}

We recall in this section the background material on simplicial sets with marking, saturated complicial sets, and on the pretensor and tensor product, $\boxtimes$ and $\otimes$.

\begin{defn}
A \emph{simplicial set with marking}\footnote{This notion is the same as \emph{stratified simplicial set} in the sense of Verity \cite{VerityComplicialAMS}, and is different from (but related to) \emph{marked simplicial set} in the sense of Lurie \cite{htt}.} is a simplicial set with a designated subset of \emph{marked} or \emph{thin} positive-dimensional simplices that includes all degenerate simplices. A \emph{map of simplicial sets with marking} is a simplicial map that preserves the marking. We denote by $\msset$ the category of simplicial sets with marking and maps of simplicial sets with marking.
\end{defn}

\subsection{The model structures on simplicial sets with marking}

The following notational conventions will be used to define saturated $N$-complicial sets and to describe the model structure for $N$-complicial sets on $\msset$. The material is mostly drawn from \cite[\textsection\textsection 2.1-2.2]{VerityComplicialI}, \cite{EmilyNotes} and \cite[\textsection 1]{or}, and we refer the reader to these references for a more detailed account.

\begin{notn}
\label{preliminarynotation}
We denote
\begin{itemize}[leftmargin=*]
\item by $\Delta[-1]$ the empty simplicial set.
    \item by $\Delta[m]$ the simplicial set with marking whose underlying simplicial set is $\Delta[m]$ and in which only degenerate simplices are marked.
        \item by $\partial\Delta[m]$ the simplicial set with marking whose underlying simplicial set is $\partial\Delta[m]$ and in which only degenerate simplices are marked.
    \item by $\Delta[m]_{t}$ the simplicial set with marking whose underlying simplicial set is $\Delta[m]$ and in which only degenerate simplices and the top $m$-simplex are marked.
    \item by $\Delta^k[m]$, for $0\leq k \leq m$, the simplicial set with marking whose underlying simplicial set is $\Delta[m]$ and in which a non-degenerate simplex is marked if and only if it contains the vertices $\{k-1,k,k+1\}\cap [m]$.
    \item by $\Delta^k[m]'$, for $0\leq k \leq m$, the simplicial set with marking obtained from $\Delta^k[m]$ by additionally marking the $(k-1)$-st and $(k+1)$-st face of $\Delta[m]$.
    \item by $\Delta^k[m]''$, for $0\leq k \leq m$, the simplicial set with marking obtained from $\Delta^k[m]'$ by additionally marking the $k$-th face of $\Delta[m]$.
    \item by $\Lambda^k[m]$, for $0\leq k \leq m$, the simplicial set with marking whose underlying simplicial set is the $k$-horn $\Lambda^k[m]$ and whose simplex is marked if and only if it is marked in $\Delta^k[m]$. 
    \item by $\eqDelta$ the simplicial set with marking whose underlying simplicial set is $\Delta[3]$, and the non-degenerate marked simplices consist of all $2$- and $3$-simplices, as well as $1$-simplices $[02]$ and $[13]$.
        \item by $\Delta[3]_\sharp$ the simplicial set with marking whose underlying simplicial set is $\Delta[3]$, and all simplices in positive dimensions are marked.
           \item by
$\DeltaThree{\ell'}{\ell}{eq}$, for $\ell,\ell'\ge-1$, the simplicial set with marking $\Delta[\ell']\star\eqDelta\star\Delta[\ell]$.
    \item by
$\DeltaThree{\ell'}{\ell}{\sharp}$, for $\ell,\ell'\ge-1$, the simplicial set with marking $\Delta[\ell']\star\Delta[3]_{\sharp}\star\Delta[\ell]$.
\end{itemize}
\end{notn}

Here, $\star$ denotes the join of simplicial sets with marking, which can be found in \cite[Observation 34]{VerityComplicialI} or \cite[Def. 3.2.5]{EmilyNotes}, and which we recall for the reader's convenience.

\begin{defn}
Given simplicial sets with marking $X$ and $Y$, the \emph{join} $X\star Y$ is a simplicial set with marking whose underlying simplicial set is the join of the underlying simplicial sets, and in which an $r$-simplex $x\star y\colon\Delta[k]\star\Delta[r-k-1]\to X\star Y$ for $-1\le k\le r$ is marked if and only if the simplex $x$ is marked in $X$ or the simplex $y$ is marked in $Y$ (or both).
\end{defn}

\begin{defn}
\label{anodynemaps}
For $N\in\mathbb N\cup\{\infty\}$, an \emph{elementary $(\infty,N)$-anodyne extension} is one of the following.
 \begin{enumerate}[leftmargin=*]
  \item  The \emph{complicial horn extension}, i.e., the canonical map
 $$\Lambda^k[m]\to \Delta^k[m]\text{ for $m\geq 1$ and $0\leq k\leq m$},$$
 which is the ordinary horn inclusion on the underlying simplicial sets.
 \item[(1')] The \emph{complicial thinness extension}, i.e., the canonical map  
$$\Delta^k[m]' \to \Delta^k[m]''\text{ for $m\geq 2$ and $0\leq k \leq m$},$$
which is the identity on the underlying simplicial set.
\item The \emph{left saturation extension}, i.e., the canonical map
$$\DeltaThree{\ell}{}{eq}  \to \DeltaThree{\ell}{}{\sharp} \text{ for $\ell\geq -1$},$$
which is the identity on the underlying simplicial set.
\item The \emph{triviality extension} map,
i.e., the canonical map
$$\Delta[p]\to \Delta[p]_t\text{ for $p>N$},$$
which is the identity on the underlying simplicial set.
\end{enumerate}
\end{defn}

\begin{rmk}
We point out that the parameter $N$ only plays a role in the triviality anodyne extension in (3). In particular, complicial horn extensions, thinness extensions and saturation anodyne extensions are $(\infty,N)$-anodyne for every $N\in\mathbb N\cup\{\infty\}$.
\end{rmk}

\begin{defn}
Let $X$ be a simplicial set with marking, and $N\in\mathbb N\cup\{\infty\}$.
\begin{enumerate}[leftmargin=*]
    \item $X$ is a \emph{complicial set}, also called a \emph{weak complicial set}, if it has the right lifting property with respect to the complicial horn anodyne extensions $\Lambda^k[m]\to \Delta^k[m]$ and the thinness anodyne extensions $\Delta^k[m]' \to \Delta^k[m]''$ for $m\geq 1$ and $0\leq k\leq m$.
    \item $X$ is a \emph{saturated complicial set} if it is a complicial set and it has the right lifting property with respect to the left  saturation anodyne extensions $\DeltaThree{\ell}{}{eq}  \to \DeltaThree{\ell}{}{\sharp}$ for $\ell\geq -1$.
    \item $X$ is a \emph{saturated $N$-complicial set} if it is a saturated complicial set and it has the right lifting property with respect to the triviality anodyne extensions $\Delta[p]\to \Delta[p]_t$ for $p>N$.
\end{enumerate}
\end{defn}

For any $N\in\mathbb N$, saturated $N$-complicial sets are a proposed model for $(\infty,N)$-categories\footnote{The case $N=\infty$ is subtle, since there are at least two different viewpoints on what an $(\infty,\infty)$-category should be.}, and we refer the reader to \cite{VerityComplicialAMS,EmilyNotes,or} for a description of the intuition behind this combinatorics.

Roughly speaking, according to the intuition that the $r$-simplices of a simplicial set with marking represent $r$-morphisms and that the marked simplices represent $r$-equivalences, we can rephrase as follows.
\begin{enumerate}[leftmargin=*]
    \item In a complicial set $r$-morphisms can be composed, and composite of $r$-equivalences is an $r$-equivalence.
    \item In a saturated complicial set $r$-equivalences satisfy the two-out-of-six property.
    \item In a saturated $N$-complicial set all $r$-morphisms are equivalences in dimension $r>N$.
\end{enumerate}

There is a model structure on $\msset$ for saturated $N$-complicial sets.

\begin{thm}[{\cite{VerityComplicialAMS,EmilyNotes,or}}]
Let $N\in\mathbb N\cup\{\infty\}$. There is a cofibrantly generated model structure on $\msset$ in which
\begin{itemize}[leftmargin=*]
    \item the cofibrations are precisely the monomorphisms;
    \item the fibrant objects are precisely the saturated $N$-complicial sets;
    \item all elementary anodyne extensions are acyclic cofibrations.
\end{itemize}
We call this model structure the model structure for $(\infty,N)$-categories, or the model structure for saturated $N$-complicial sets, we denote it by $\msset_{(\infty,N)}$, and we call the acyclic cofibrations $(\infty,N)$-acyclic cofibrations.
\end{thm}

\begin{rmk}
As discussed in \cite[Example 21]{VerityComplicialI},
the generating cofibrations for the model structure for $(\infty,N)$-categories are \begin{itemize}[leftmargin=*]
    \item the \emph{boundary inclusions}
\[\partial\Delta[m]\to\Delta[m]\text{ for }m\ge0,\]
\item and the \emph{marking inclusions}
\[\Delta[m]\to\Delta[m]_t\text{ for }m\ge1.\]
\end{itemize}
\end{rmk}

We mentioned that, by construction, all left saturation extensions $\DeltaThree{\ell}{}{eq}  \to \DeltaThree{\ell}{}{\sharp}$ for $\ell\geq -1$ are acyclic cofibrations. In fact, even the saturation extensions of the more general form $\DeltaThree{\ell'}{\ell}{eq}  \to \DeltaThree{\ell'}{\ell}{\sharp}$ for $\ell,\ell'\geq -1$ are acyclic cofibrations.

\begin{lem}
\label{generalsaturation}
The \emph{saturation extension} \[\DeltaThree{\ell'}{\ell}{eq}  \to \DeltaThree{\ell'}{\ell}{\sharp}\quad\text{ for }\quad\ell,\ell'\geq -1\] is acyclic cofibration.
\end{lem}

\begin{proof}
The saturation extensions $\DeltaThree{\ell'}{\ell}{eq}  \to \DeltaThree{\ell'}{\ell}{\sharp}$ have the left lifting property with respect to all saturated $N$-complicial sets, as shown \cite[\textsection D.7]{RiehlVerityBook}, and since they are isomorphisms on the underlying simplicial sets they must also have the right lifting property with respect to all fibrations between saturated $N$-complicial sets. We then conclude that they are acyclic cofibrations as an instance of \cite[Lemma 7.14]{JT}.
\end{proof}

\begin{prop}[{\cite[Lemma~1.8]{FundamentalPushouts}}]
\label{leftQuillen}
Let $\cM$ be a model category.
A left adjoint functor $F\colon\msset_{(\infty,N)}\to\cM$ is left Quillen if and only if it respects cofibrations and sends
all elementary anodyne extensions to weak equivalences
of $\cM$.
\end{prop}

\subsection{Pretensor and tensor product of simplicial sets with marking}
Inspired\footnote{
In \cite[Observation 62]{VerityComplicialI},
Verity states the relationship between the
Crans--Gray tensor product of $\omega$-categories and 
the tensor product of simplicial sets with marking, using the fact that $\omega$-categories (in the form of strict complicial sets) form a reflective subcategory of simplicial sets with marking. Given two $\omega$-categories, their Crans--Gray tensor product can be obtained by reflecting their tensor product as simplicial sets with marking.
}
by the Crans--Gray tensor product of $\omega$-categories from \cite{GrayFormalCategoryTheory, CransOmega}, which can be thought as strict $\infty$-categories, 
Verity defined two models of Gray tensor products of simplicial sets with marking: the pretensor $\boxtimes$ and the tensor $\otimes$. In this paper, we will work with the definition of the tensor product $\otimes$, while the pretensor product $\boxtimes$ plays a more indirect role. For completeness, we recall both definitions.

\begin{notn}[{\cite[Notation 5]{VerityComplicialAMS}}]
For any $p,q\ge0$, 
\begin{itemize}[leftmargin=*]
    \item the \emph{degeneracy partition operator} is the map in $\Delta$
\[\Pi_1^{p,q}\colon[p+q]\to[p]\quad\text{ and }\quad\Pi_2^{p,q}\colon[p+q]\to[q]\]
defined by
\[i\mapsto\left\{
\begin{array}{cc}
    i & \text{ if } i\leq p\\
    p & \text{ if } i>p
\end{array}\right.\quad\text{ and }\quad
i\mapsto\left\{
\begin{array}{cc}
    0 & \text{ if } i<p\\
    i-p & \text{ if }i\geq p
\end{array}\right.\]
\item the \emph{face partition operator} is the map in $\Delta$
\[\amalg_1^{p,q}\colon[p]\to[p+q]\quad\text{ and }\quad\amalg_2^{p,q}\colon[q]\to[p+q]\]
defined by
\[i\mapsto i\quad\text{ and }\quad
i\mapsto p+i.\]
\end{itemize}
\end{notn}

\begin{rmk}
As explained in \cite[\textsection 1.6]{VerityComplicialAMS}, any non-degenerate $r$-simplex of $\Delta[r]\to\Delta[p]\times \Delta[q]$ can be pictured as a path of length $r$ in a rectangular grid of size $p\times q$. According to this interpretation, the $(p+q)$-simplex given by $(\Pi_1^{p,q}, \Pi_2^{p,q})\colon\Delta[p+q]\to\Delta[p]\times\Delta[q]$ is the path with \lq\lq{}first all to the right, then all up\rq\rq{}, as shown in the following picture for $p=3$ and $q=2$.
\begin{center}
    \begin{tikzpicture}
    \def\l{1cm}
\foreach \x in {0,1,2,3}
\foreach \y in {0,1,2}
    \draw[fill] (\x*\l, \y*\l) circle (1pt) node (a\x\y){};
 \draw (a00)--(a10);   
 \draw (a10)--(a20);
 \draw (a20)--(a30);
 \draw (a30)--(a31);
 \draw (a31)--(a32);
\end{tikzpicture}
\end{center}

\end{rmk}

\begin{defn}[{\cite[Def.~135]{VerityComplicialI}}]
Given simplicial sets with marking $X$ and $Y$, the \emph{pretensor} $X\boxtimes Y$ is formed by taking the product of
underlying simplicial sets and endowing it with a marking under which a non-degenerate $r$-simplex $(x, y) \colon\Delta[r]\to
X \times Y$ is marked if either
\begin{itemize}[leftmargin=*]
    \item it is a \emph{mediator}, i.e., there exists $0 <k<r$ and $(r-1)$-simplices $x'\colon\Delta[r-1]\to X$ and $y'\colon\Delta[r-1]\to Y$ such that $x = s_{k-1}x'=x'\circ s^{k-1}$ and $y = s_ky'=y'\circ s^k$.
    \item it is a \emph{crushed cylinder}, i.e., there exists a partition $p, q$ of $r=p+q$ and simplices $x'\colon\Delta[p]\to X$ and $y'\colon\Delta[q]\to Y$ such that
$x = x'\circ\Pi^{p,q}_1$ and $y = y'\circ\Pi^{p,q}_2$, and
either the simplex $x'$ is marked in $X$ or the simplex $y'$ is marked in $Y$ (or both).
\end{itemize}
\end{defn}

It is proven in \cite[Lemma 142]{VerityComplicialI} that $\boxtimes$ is a bifunctor that preserves colimits in each variable. We then obtain the following adjunctions. Regarding the terminology of lax and oplax, we follow the same convention as e.g.\ \cite{LackIcons, AraLucas}. 

\begin{prop}[{\cite[Cor.~144]{VerityComplicialAMS}}]
For any simplicial set with marking $S$ there are adjunctions
\[-\boxtimes S\colon\msset\rightleftarrows\msset\colon[S,-]_{oplax}\]
and
\[S\boxtimes -\colon\msset_{(\infty,N)}\rightleftarrows\msset_{(\infty,N)}\colon[S,-]_{lax}.\]
\end{prop}

However, the pretensor $\boxtimes$ is not associative, so it cannot be used to build a monoidal structure on $\msset$. For this purpose, one can instead consider the tensor product $\otimes$ (which however does not preserve colimits).

\begin{defn}[{\cite[Def.~128]{VerityComplicialAMS}}]
Given simplicial sets with marking $X$ and $Y$, the \emph{tensor} $X\otimes Y$ is formed by taking the product of
underlying simplicial sets and endowing it with a marking under which a non-degenerate $r$-simplex $(x, y)\colon\Delta[r]\to
X \times Y$ is marked if for each $p,q\ge0$ the partition $r=p+q$ \emph{cleaves} the simplex $(x,y)$, i.e., the $p$-simplex
$x \circ\amalg^{p,q}_1$ is marked in $X$ or the $q$-simplex $ y\circ\amalg^{p,q}_2$ is marked in $Y$.
\end{defn}

Pretensor and tensor are equivalent in the following sense.

\begin{prop}[{\cite[Lemma 149]{VerityComplicialAMS}}]
\label{EquivalentTensors}
For any simplicial sets with marking $X$ and $Y$ the canonical inclusion
\[X\boxtimes Y\hookrightarrow X\otimes Y\]
is an $(\infty,N)$-acyclic cofibration
for any $N\in\mathbb N\cup\{\infty\}$. In particular there is an objectwise weak equivalence
\[-\boxtimes -\simeq-\otimes -\colon\msset_{(\infty,N)}\times\msset_{(\infty,N)}\to\msset_{(\infty,N)}.\]
\end{prop}

To highlight the difference between the pretensor and the tensor, we briefly discuss an example. We refer the reader to \cite[\textsection 6.3]{VerityComplicialAMS} for a deeper treatment and for more details and examples.
\begin{ex}
We consider the case of $X=\Delta[2]_t$ and $Y=\Delta[1]$.
\begin{itemize}[leftmargin=*]
\item the simplex $\Delta[2]\to\Delta[2]_t\times\Delta[1]$ depicted as
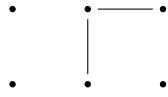
\begin{center}
    \begin{tikzpicture}
    \def\l{1cm}
\foreach \x in {0,1,2}
\foreach \y in {0,1}
    \draw[fill] (\x*\l, \y*\l) circle (1pt) node (a\x\y){};
 \draw (a10)--(a11);
 \draw (a11)--(a21);
\end{tikzpicture}
\end{center}
is a mediator, and is therefore marked in both $\Delta[2]_t\boxtimes\Delta[1]$ and $\Delta[2]_t\otimes\Delta[1]$.
\item the simplex $\Delta[3]\to\Delta[2]_t\times\Delta[1]$ depicted as
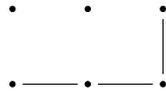
\begin{center}
    \begin{tikzpicture}
    \def\l{1cm}
\foreach \x in {0,1,2}
\foreach \y in {0,1}
    \draw[fill] (\x*\l, \y*\l) circle (1pt) node (a\x\y){};
 \draw (a00)--(a10);
  \draw (a10)--(a20);
 \draw (a20)--(a21);
\end{tikzpicture}
\end{center}
is a crushed cylinder, and is therefore marked in both $\Delta[2]_t\boxtimes\Delta[1]$ and $\Delta[2]_t\otimes\Delta[1]$.
\item the simplex $\Delta[2]\to\Delta[2]_t\times\Delta[1]$ depicted as
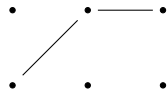
\begin{center}
    \begin{tikzpicture}
    \def\l{1cm}
\foreach \x in {0,1,2}
\foreach \y in {0,1}
    \draw[fill] (\x*\l, \y*\l) circle (1pt) node (a\x\y){};
 \draw (a00)--(a11);
 \draw (a11)--(a21);
\end{tikzpicture}
\end{center}
is cleaved by every partition, and is therefore marked in  $\Delta[2]_t\otimes\Delta[1]$, but it is not marked in $\Delta[2]_t\boxtimes\Delta[1]$.
\item the simplex $\Delta[2]\to\Delta[2]_t\times\Delta[1]$ depicted as
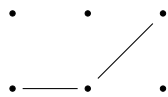
\begin{center}
    \begin{tikzpicture}
    \def\l{1cm}
\foreach \x in {0,1,2}
\foreach \y in {0,1}
    \draw[fill] (\x*\l, \y*\l) circle (1pt) node (a\x\y){};
 \draw (a00)--(a10);
 \draw (a10)--(a21);
\end{tikzpicture}
\end{center}
is cleaved by the partitions $(2,0)$ and $(0,2)$, but not by the partition $(1,1)$, and is therefore not marked neither in $\Delta[2]_t\boxtimes\Delta[1]$ nor in $\Delta[2]_t\otimes\Delta[1]$.
\end{itemize}
\end{ex}

\section{The main theorem}

The main result is the following.

\begin{thm}
\label{pretensorquillen}
Let $N\in\mathbb N\cup\{\infty\}$. 
For any simplicial set with marking $S$ the adjunction
\[-\boxtimes S\colon\msset_{(\infty,N)}\rightleftarrows\msset_{(\infty,N)}\colon[S,-]_{oplax}\]
is a Quillen pair. In particular, the functor
\[-\boxtimes S\colon\msset_{(\infty,N)}\to\msset_{(\infty,N)}\]
is homotopical.
\end{thm}

The theorem admits many essentially equivalent reformulations or direct consequences, which we collect as corollaries.

Using \cref{EquivalentTensors} we obtain the following corollary.

\begin{cor}
\label{tensorhomotopical}
Let $N\in\mathbb N\cup\{\infty\}$. For any simplicial set with marking $S$ the functor
\[-\otimes S\colon\msset_{(\infty,N)}\to\msset_{(\infty,N)}\]
is homotopical.
\end{cor}
The statement can then be strengthened as follows.

\begin{cor}
\label{tensorhomotopical}
Let $N\in\mathbb N\cup\{\infty\}$. The functor
\[-\otimes-\colon\msset_{(\infty,N)}\times\msset_{(\infty,N)}\to\msset_{(\infty,N)}\]
is homotopical.
\end{cor}

\begin{lem}
\label{opequivalence}
Let $f\colon X\to Y$ be a map of simplicial sets with marking. Then $f$ is a weak equivalence in the model structure for saturated $N$-complicial sets if and only if $f^{\op}$ is one.
\end{lem}

\begin{proof}[Proof of \cref{opequivalence}]
We argue that $(-)^{\op}$ is left Quillen, so in particular homotopical, and hence respects weak equivalences. Given the canonical isomorphism $(X^{\op})^{\op}\cong X$ from \cite[Observation 38]{VerityComplicialI}, we also obtain that $(-)^{\op}$ reflects weak equivalences, concluding the proof.

To see that $(-)^{\op}\colon\msset_{(\infty,N)}\to\msset_{(\infty,N)}$ is a left Quillen functor, we observe the following.
\begin{enumerate}[leftmargin=*]
    \item[(0)] Since $(-)^{\op}$ is an isomorphism, if $X\to Y$ is a monomorphism, then $X^{\op}\to Y^{\op}$ is a monomorphism, so $(-)^{\op}$ preserves cofibrations.
    \item By \cite[Observation 157]{VerityComplicialAMS}, for for $m\geq 0$ and $0\leq k\leq m$ the map $\Lambda^k[m]^{\op}\to\Delta^k[m]^{\op}$ is the map $\Lambda^{m-k}[m]\to \Delta^{m-k}[m]$, which is a weak equivalence in the model structure for saturated $N$-complicial sets. In particular, $(-)^{\op}$ sends complicial horn extensions to weak equivalences.
    \item By \cite[Observation 125]{VerityComplicialAMS}, for $m\geq 0$ and $0\leq k\leq m$ the map $\Delta^k[m]'^{\op}\to\Delta^k[m]''^{\op}$ is the map $\Delta^{m-k}[m]'\to\Delta^{m-k}[m]''$, which is a weak equivalence in the model structure for saturated $N$-complicial sets. In particular, $(-)^{\op}$ sends thinness extensions to weak equivalences.
    \item By \cite[Observation 107]{VerityComplicialAMS}, for $p>N$ the map $\Delta[p]^{\op}\to\Delta[p]_t^{\op}$ is $\Delta[p]\to \Delta[p]_t$, which is a weak equivalence in the model structure for saturated $N$-complicial sets. In particular $(-)^{\op}$ sends triviality extensions for $p>N$ to weak equivalences.
    \item For $\ell\ge-1$, one can use \cite[Observation 36]{VerityComplicialI} to show that the map $\DeltaThree{\ell}{}{eq}^{\op}\to\DeltaThree{\ell}{}{\sharp}^{\op}$ is the map $\DeltaThree{}{\ell}{eq}\to\DeltaThree{}{\ell}{\sharp}$, which was shown in \cref{generalsaturation} to be a weak equivalence in the model structure for saturated $N$-complicial sets. In particular, $(-)^{\op}$ sends left saturation extensions to weak equivalences.
\end{enumerate}
By \cref{leftQuillen}, we then conclude that $(-)^{\op}$ is a left Quillen functor, as desired.
\end{proof}

\begin{proof}[Proof of \cref{tensorhomotopical}]
We already know from \cref{pretensorquillen} that the functor $\otimes$ respects weak equivalences in the first variable, and we now check that it respects weak equivalences in the second variable, too.
If $X\to Y$ is a weak equivalence, by \cref{opequivalence} the map $X^{\op}\to Y^{\op}$ is a weak equivalence.
By \cref{pretensorquillen} the map $X^{\op}\boxtimes S^{\op}\to Y^{\op}\boxtimes S^{\op}$ is a weak equivalence. By \cref{EquivalentTensors} the map $X^{\op}\otimes S^{\op}\to Y^{\op}\otimes S^{\op}$, which is by \cite[Lemma 131]{VerityComplicialAMS} the map $(S\otimes X)^{\op}\to (S\otimes Y)^{\op}$, is a weak equivalence. Using \cref{opequivalence}, the map $S\otimes X\to S\otimes Y$ is then a weak equivalence, as desired.
\end{proof}

Using again \cref{EquivalentTensors} we obtain the following corollary.

\begin{cor}
Let $N\in\mathbb N\cup\{\infty\}$. The functor
\[-\boxtimes-\colon\msset_{(\infty,N)}\times\msset_{(\infty,N)}\to\msset_{(\infty,N)}\]
is homotopical.
\end{cor}

Since cofibrations in the model category $\msset_{(\infty,N)}$ are checked on the underlying simplicial set, we obtain the following corollary.

\begin{cor}
\label{pretensorbiquillen}
Let $N\in\mathbb N\cup\{\infty\}$. 
The functor
\[-\boxtimes-\colon\msset_{(\infty,N)}\times\msset_{(\infty,N)}\to\msset_{(\infty,N)}\]
is a left Quillen bifunctor.
In particular, for any simplicial set with marking $S$ the adjunction
\[S\boxtimes -\colon\msset_{(\infty,N)}\rightleftarrows\msset_{(\infty,N)}\colon[S,-]_{lax}\]
is a Quillen pair.
\end{cor}

\subsection{The formal part of the proof}

In this subsection we prove \cref{pretensorquillen} building on existing work of the third-named author and on a technical fact (\cref{pushouttensorsaturation}) whose proof will be postponed until the last subsection.
\begin{prop}
\label{pushoutpretensoranodynemarking}
Let $N\in\mathbb N\cup\{\infty\}$. For any $m\ge0$ the pushout-pretensor
\[(J\boxtimes \Delta[m])\aamalg{I\boxtimes \Delta[m]}(I\boxtimes \Delta[m]_t)\to J\boxtimes \Delta[m]_t\]
of an $(\infty,N)$-elementary anodyne extension $I\to J$ with the canonical map $\Delta[m]\hookrightarrow\Delta[m]_t$
is an $(\infty,\infty)$-acyclic cofibration.
\end{prop}

\begin{proof}
By \cite[Lemma 140]{VerityComplicialAMS} the pushout-pretensor of two \emph{entire} maps in the sense of \cite[Notation 100]{VerityComplicialAMS}, namely maps that are an isomorphism on the underying simplicial sets,
is an isomorphism. Hence, in particular the pushout-pretensor of a complicial thinness extension $\Delta^k[m]'\hookrightarrow\Delta^k[m]''$ with the canonical map $\Delta[m]\hookrightarrow\Delta[m]_t$ is an isomorphism. Moreover, it is explained in the proof of \cite[Lemma 169]{VerityComplicialAMS} that the pushout-pretensor of a complicial horn extension $\Lambda^k[m]\hookrightarrow\Delta^k[m]$ with the canonical map $\Delta[m]\hookrightarrow\Delta[m]_t$ is an $(\infty,\infty)$-acyclic cofibration.
\end{proof}

\begin{prop}
\label{pushoutpretensoranodyneboundary}
Let $N\in\mathbb N\cup\{\infty\}$. For any $m\ge0$ the pushout-pretensor
\[(J\boxtimes \partial\Delta[m])\aamalg{I\boxtimes \partial\Delta[m]}(I\boxtimes \Delta[m])\to J\boxtimes \Delta[m]\]
of an elementary $(\infty,N)$-anodyne extension $I\to J$ with a boundary inclusion $\partial\Delta[m]\hookrightarrow\Delta[m]$
is an $(\infty,N)$-acyclic cofibration.
\end{prop}

\begin{proof}
We treat each type of elementary anodyne extension.
\begin{enumerate}[leftmargin=*]
\item It is explained in the proof of \cite[Lemma 143]{VerityComplicialAMS} that the
pushout-pretensor of an thinness elementary anodyne extension $\Delta^k[m]'\hookrightarrow\Delta^k[m]''$ with a boundary inclusion is an $(\infty,\infty)$-acyclic cofibration.
\item It is explained in the proof of \cite[Lemma 169]{VerityComplicialAMS} that the
pushout-pretensor of a complicial horn $\Lambda^k[m]\hookrightarrow\Delta^k[m]$ extension with a boundary inclusion is an $(\infty,\infty)$-acyclic cofibration.
\item We will show in \cref{pushouttensorsaturation} that the pushout-tensor of a left saturation extension $\DeltaThree{\ell}{}{eq}\to\DeltaThree{\ell}{}{\sharp}$ with a boundary inclusion is an $(\infty,\infty)$-acyclic cofibration. By \cref{EquivalentTensors} (together with the fact that the pushout that needs to be analyzed is in fact a homotopy pushout), this implies that also that the pushout-pretensor of a left saturation extension $\DeltaThree{\ell}{}{eq}\hookrightarrow\DeltaThree{\ell}{}{\sharp}$ with a boundary inclusion is an $(\infty,\infty)$-acyclic cofibration.
\item We will show in \cref{pushoutpretensortriviality} that the pushout-pretensor of a triviality extension $\Delta[p]\to\Delta[p]_t$ for $p>N$ with a boundary inclusion is an $(\infty,N)$-acyclic cofibration.\qedhere
\end{enumerate}
\end{proof}

The proof above made use of the following two propositions.

\begin{prop}
\label{pushoutpretensortriviality}
Let $N\in\mathbb N\cup\{\infty\}$. For any $m\ge0$ and $p>N$ the pushout-pretensor
\[( \Delta[p]_t\otimes \partial \Delta[m]) \aamalg{\Delta[p] \otimes\partial \Delta[m] } (\Delta[p]\otimes \Delta[m]) \to \Delta[p]_t \otimes \Delta[m]\]
of an $(\infty,N)$-triviality anodyne extension
$\Delta[p]\to\Delta[p]_t$
with a boundary inclusion $\partial\Delta[m]\hookrightarrow\Delta[m]$ is an $(\infty,N)$-acyclic cofibration.
\end{prop}
\begin{proof}
The simplicial sets with marking $( \Delta[p]_t\otimes \partial \Delta[m]) \amalg_{\Delta[p] \otimes\partial \Delta[m] } (\Delta[p]\otimes \Delta[m])$ and $\Delta[p]_t \otimes \Delta[m]$ have the same underlying simplicial set, isomorphic to $\Delta[p] \times \Delta[m]$. We observe that they also have the same set of marked $r$-simplices for $r<p$. Indeed, the set of marked simplices in dimension $r<p$ is already contained in $\partial\Delta[p]\otimes\Delta[m]$. Moreover, for any $r$-simplex $\sigma\colon\Delta[r]\to\Delta[p]_t \otimes \Delta[m]$ for $r\ge p$ we can consider the map of simplicial sets
\[\Delta[r]\to( \Delta[p]_t\otimes \partial \Delta[m]) \aamalg{\Delta[p] \otimes\partial \Delta[m] } (\Delta[p]\otimes \Delta[m]),\]
and realize $\Delta[p]_t \otimes \Delta[m]$ as the pushout along the union of many triviality anodyne extensions:
\[\begin{tikzcd}
\coprod\limits_{\sigma}\Delta[r] \arrow[d]\arrow[r] & \coprod\limits_{\sigma}\Delta[r]_{t}\arrow[d]\\
( \Delta[p]_t\otimes \partial \Delta[m]) \aamalg{\Delta[p] \otimes\partial \Delta[m] } (\Delta[p]\otimes \Delta[m]) \arrow[r]&\Delta[p]_t \otimes \Delta[m]
\end{tikzcd}\]
In particular, the inclusion in question is an $(\infty,N)$-acyclic cofibration, as desired.
\end{proof}

\begin{prop}
\label{pushouttensorsaturation}
Let $N\in\mathbb N\cup\{\infty\}$. For any $m\ge0$ and $\ell\geq -1$ the pushout-tensor
\[(\DeltaThree{\ell}{}{\sharp}\otimes\partial \Delta[m]  ) \aamalg{\DeltaThree{\ell}{}{eq} \otimes \partial \Delta[m]} (\DeltaThree{\ell}{}{eq}\otimes \Delta[m]  ) \to  \DeltaThree{\ell}{}{\sharp}\otimes \Delta[m]\]
of a saturation anodyne extension
$\DeltaThree{\ell}{}{eq}\to \DeltaThree{\ell}{}{\sharp}$
with a boundary inclusion $\partial\Delta[m]\hookrightarrow\Delta[m]$ is an $(\infty,\infty)$-acyclic cofibration.
\end{prop}

The proof of this proposition is postponed until the last section.

We can now prove the theorem.

\begin{proof}[Proof of \cref{pretensorquillen}]

To see that $-\boxtimes S\colon\msset_{(\infty,N)}\to\msset_{(\infty,N)}$ is a left Quillen functor, we observe the following.
\begin{itemize}[leftmargin=*]
    \item Since the underlyng simplicial set of the pretensor of simplicial sets with marking is product of the underlying simplicial sets, if $X\to Y$ is a monomorphism, then $X\boxtimes S\to Y\boxtimes S$ is a monomorphism at the level of underlying simplicial sets. In particular, $-\boxtimes S$ preserves cofibrations.
    \item If $I\to J$ is an elementary anodyne extension, the map $I\boxtimes S\to J\boxtimes S$ can be written as the pushout product
    \[J\boxtimes\Delta[-1]\aamalg{I\boxtimes \Delta[-1]} I \boxtimes S\to J\boxtimes S.\]
    It can then be deduced from \cref{pushoutpretensoranodyneboundary,pushoutpretensoranodynemarking}
    using the compatibility of pushouts and pretensor product with colimits that the functor $-\boxtimes S$ sends all elementary anodyne extensions to weak equivalences.
\end{itemize}
By \cref{leftQuillen}, we then conclude that the functor $-\boxtimes S$ is a left Quillen functor, as desired.
\end{proof}

\subsection{Proof of \cref{pushouttensorsaturation}}

In this subsection we provide the last missing verification.

\begin{rmk}
\label{simplicesofS0}
A non-degenerate $r$-simplex $\sigma\colon \Delta[r]\to \Delta[\ell+4]\times \Delta[m]$ is marked in $\DeltaThree{\ell}{}{eq}\otimes \Delta[m]$ (resp. $\DeltaThree{\ell}{}{\sharp}\otimes \Delta[m]$) if and only if\footnote{This reasoning is inspired by \cite[Lemma 129]{VerityComplicialAMS}.}
\begin{itemize}[leftmargin=*] 
 \item the second projection $\pr_2\sigma$ is degenerate (in particular there exists a maximal $1\leq h\leq r$ such that $\pr_2\sigma(h-1)=\pr_2\sigma(h)$ and we call this $h$ the \emph{degeneracy index} of $\sigma$), and
 \item the partition face $\amalg_1^{h, r-h}$ of the first projection $(\pr_1\sigma)\circ\amalg_1^{h,r-h}$ is marked in $\DeltaThree{\ell}{}{eq}$ (resp. $\DeltaThree{\ell}{}{\sharp}$).
\end{itemize}
Informally speaking, the degeneracy index $h$ of a simplex $\sigma$ is the maximal value for which $\sigma(h)$ is the final point of a horizontal piece in the path that describes the simplex $\sigma$.
\end{rmk}

\begin{proof}[Proof of \cref{pushouttensorsaturation}]
For simplicity of notation, we write
\[S_0:=(\DeltaThree{\ell}{}{\sharp}\otimes\partial \Delta[m]  ) \aamalg{\DeltaThree{\ell}{}{eq} \otimes \partial \Delta[m]} (\DeltaThree{\ell}{}{eq}\otimes \Delta[m]  ).\]
and we show by induction on $l$ that the map $S_0\to\DeltaThree{\ell}{}{\sharp}\otimes \Delta[m]$ is an acyclic cofibration for any $m\geq 0$ and any $\ell\geq -1$.

The simplicial sets with marking $S_0$ and $\DeltaThree{\ell}{}{\sharp}\otimes \Delta[m]$ have the same underlying simplicial set, isomorphic to $  \Delta[\ell+4]\times \Delta[m]$.
By \cref{simplicesofS0}, the $r$-simplices of $\DeltaThree{\ell}{}{\sharp}\otimes \Delta[m]$ that are not marked in $S_0$ are then characterized as follows:
An $r$-simplex is marked in $\DeltaThree{\ell}{}{\sharp}\otimes \Delta[m]$ and not in $S_0$ if and only if
\begin{itemize}[leftmargin=*]
    \item the second projection $\pr_2\sigma$ is surjective, so in particular $r\geq m$ and
    \[(\pr_2\sigma)\circ\amalg_2^{h,r-h}=\id_{\Delta[r-h]}\colon\Delta[r-h]\to\Delta[r-h],\]
    and
    \item the partition face $\amalg_1^{h,r-h}$ of the first component 
     $\pr_1\sigma$ is of the form \[(\pr_1\sigma)\circ\amalg_1^{h,r-h}=\sigma'\star \sigma''\colon \Delta[h-2]\star \Delta[1]\to \Delta[\ell]\star \Delta[3]\]
     with $\sigma''\in\{[01], [03], [12], [23]\}$ and $\sigma'$ non-degenerate.
\end{itemize}

We will now mark all simplices $\sigma$ marked in $\DeltaThree{\ell}{}{\sharp}\otimes \Delta[m]$ and not in $S_0$ by constructing a sequence of entire acyclic cofibrations
\[S_0\hookrightarrow S_1\hookrightarrow S_2\hookrightarrow S_3\hookrightarrow S_4\hookrightarrow S_5\hookrightarrow S_6\cong\DeltaThree{\ell}{}{\sharp}\otimes \Delta[m],\]
which will prove the lemma. More precisely, we will mark
\begin{enumerate}[leftmargin=*]
    \item in $S_1$ exactly all simplices $\sigma$ marked in $\DeltaThree{\ell}{}{\sharp}\otimes \Delta[m]$ and not in $S_0$ that are contained in a copy of $\DeltaThree{\ell-1}{}{\sharp}\otimes \Delta[m]\hookrightarrow\DeltaThree{\ell}{}{\sharp}\otimes \Delta[m]$ by means of induction hypothesis if $\ell>-1$.
     \item in $S_2$ all simplices $\sigma$ marked in $\DeltaThree{\ell}{}{\sharp}\otimes \Delta[m]$ and not in $S_1$ with $\sigma''\in\{[03],[23]\}$ (as well as other simplices) by means of saturation extensions.   The generic simplex $\sigma$ that is being marked in $S_2$ can be depicted as follows.
     
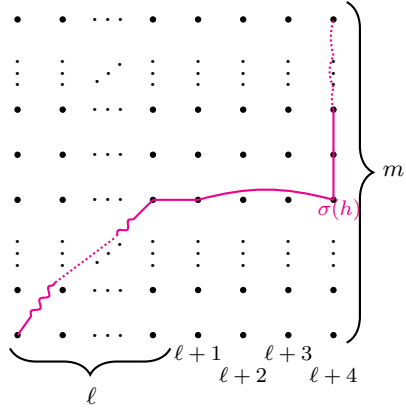
\begin{center}
\begin{tikzpicture}
\def\l{0.6cm}
\draw[fill] (0,0) circle (1pt) node (a00){};
\draw[fill] (\l, 0) circle (1pt) node (a10){};
\draw (2*\l, 0) node(a20){$\ldots$};
\draw[fill] (3*\l, 0) circle (1pt) node (a30){};
\draw[fill] (0,\l) circle (1pt) node (a01){};
\draw[fill] (\l, \l) circle (1pt) node (a11){};
\draw (2*\l, \l) node(a21){$\ldots$};
\draw[fill] (3*\l, \l) circle (1pt) node (a31){};
\draw (0,2*\l) node[inner sep=0, outer sep=0] (a02){$\vdots$};
\draw (\l,2*\l) node[inner sep=0, outer sep=0]  (a12){$\vdots$};
\draw (2*\l,2*\l) node[inner sep=0, outer sep=0]  (a22){$\iddots$};
\draw (3*\l,2*\l) node[inner sep=0, outer sep=0]  (a32){$\vdots$};
\draw[fill] (0,3*\l) circle (1pt) node (a03){};
\draw[fill] (\l, 3*\l) circle (1pt) node (a13){};
\draw (2*\l, 3*\l) node(a23){$\ldots$};
\draw[fill] (3*\l, 3*\l) circle (1pt) node (a33){};
\draw[fill] (4*\l, 0) circle (1pt) node (a40){} node[below](l+1){{\scriptsize $\ell+1$}};
\draw[fill] (5*\l, 0) circle (1pt) node (a50){} node[below, yshift=-0.3cm](l+2){{\scriptsize $\ell+2$}};
\draw[fill] (6*\l, 0) circle (1pt) node (a60){} node[below](l+3){{\scriptsize $\ell+3$}};
\draw[fill] (7*\l, 0) circle (1pt) node (a70){} node[below, yshift=-0.3cm](l+4){{\scriptsize $\ell+4$}};
\draw[fill] (4*\l, \l) circle (1pt) node (a41){};
\draw[fill] (5*\l, \l) circle (1pt) node (a51){};
\draw[fill] (6*\l, \l) circle (1pt) node (a61){};
\draw[fill] (7*\l, \l) circle (1pt) node (a71){};
\draw (4*\l,2*\l) node[inner sep=0, outer sep=0] (a42){$\vdots$};
\draw (5*\l,2*\l) node[inner sep=0, outer sep=0]  (a52){$\vdots$};
\draw (6*\l,2*\l) node[inner sep=0, outer sep=0]  (a62){$\vdots$};
\draw (7*\l,2*\l) node[inner sep=0, outer sep=0]  (a72){$\vdots$};
\draw[fill] (4*\l, 3*\l) circle (1pt) node (a43){};
\draw[fill] (5*\l, 3*\l) circle (1pt) node (a53){};
\draw[fill] (6*\l, 3*\l) circle (1pt) node (a63){};
\draw[fill] (7*\l, 3*\l) circle (1pt) node (a73){};
\draw[fill] (0,4*\l) circle (1pt) node (a04){};
\draw[fill] (\l, 4*\l) circle (1pt) node (a14){};
\draw (2*\l, 4*\l) node(a24){$\ldots$};
\draw[fill] (3*\l, 4*\l) circle (1pt) node (a34){};
\draw[fill] (4*\l, 4*\l) circle (1pt) node (a44){};
\draw[fill] (5*\l, 4*\l) circle (1pt) node (a54){};
\draw[fill] (6*\l, 4*\l) circle (1pt) node (a64){};
\draw[fill] (7*\l, 4*\l) circle (1pt) node (a74){};
\draw[fill] (0,5*\l) circle (1pt) node (a05){};
\draw[fill] (\l, 5*\l) circle (1pt) node (a15){};
\draw (2*\l, 5*\l) node(a25){$\ldots$};
\draw[fill] (3*\l, 5*\l) circle (1pt) node (a35){};
\draw[fill] (4*\l, 5*\l) circle (1pt) node (a45){};
\draw[fill] (5*\l, 5*\l) circle (1pt) node (a55){};
\draw[fill] (6*\l, 5*\l) circle (1pt) node (a65){};
\draw[fill] (7*\l, 5*\l) circle (1pt) node (a75){};
\draw (0,6*\l) node[inner sep=0, outer sep=0] (a06){$\vdots$};
\draw (\l,6*\l) node[inner sep=0, outer sep=0]  (a16){$\vdots$};
\draw (2*\l,6*\l) node[inner sep=0, outer sep=0]  (a26){$\iddots$};
\draw (3*\l,6*\l) node[inner sep=0, outer sep=0]  (a36){$\vdots$};
\draw (4*\l,6*\l) node[inner sep=0, outer sep=0] (a46){$\vdots$};
\draw (5*\l,6*\l) node[inner sep=0, outer sep=0]  (a56){$\vdots$};
\draw (6*\l,6*\l) node[inner sep=0, outer sep=0]  (a66){$\vdots$};
\draw (7*\l,6*\l) node[inner sep=0, outer sep=0]  (a76){$\vdots$};
\draw[fill] (0,7*\l) circle (1pt) node (a07){};
\draw[fill] (\l, 7*\l) circle (1pt) node (a17){};
\draw (2*\l, 7*\l) node(a27){$\ldots$};
\draw[fill] (3*\l, 7*\l) circle (1pt) node (a37){};
\draw[fill] (4*\l, 7*\l) circle (1pt) node (a47){};
\draw[fill] (5*\l, 7*\l) circle (1pt) node (a57){};
\draw[fill] (6*\l, 7*\l) circle (1pt) node (a67){};
\draw[fill] (7*\l, 7*\l) circle (1pt) node (a77){};
\draw[decorate,decoration={snake,amplitude=.4mm,segment length=2mm,pre length=0.3cm}, thick, magenta] (a00.center)--(0.5,0.7);
\draw (a33.center)+(-0.5,-0.5) node[inner sep=0, outer sep=0](targ){};
\draw[densely dotted, magenta, thick] (0.5, 0.7)--(targ);
\draw[decorate,decoration={snake,amplitude=.4mm,segment length=2mm,post length=0.3cm}, thick, magenta] (targ)--(a33.center);
\draw[thick, magenta] (a33.center)--(a43.center);
\draw[thick, magenta, bend left=15] (a43.center) to (a73.center);
\draw[thick, magenta] (a73.center) to (a74.center);
\draw[thick, magenta] (a74.center) to (a75.center);
\draw[thick, magenta, densely dotted, bend left=15] (a75.center) to ++(0,0.5*\l);
\draw[thick, magenta, densely dotted, bend left=25](a75.center) ++(0,0.5*\l) to ++(0,0.7*\l);
\draw[thick, magenta, densely dotted, bend left=15](a75.center) ++(0,1.2*\l)to (a77.center);
\draw[magenta] (a73.south east) node[xshift=-0.05cm](labelh){{\scriptsize $\sigma(h)$}};
\draw[decorate, decoration= {brace, raise=5pt, amplitude=10pt, mirror}, thick] (a00)+(-0.1cm,0)--node[midway,yshift=-0.8cm, xshift=0.2cm]  {\footnotesize $\ell$ }(a30)--++(0.5cm,0);
\draw[decorate, decoration= {brace, raise=5pt, amplitude=10pt, mirror}, thick] (a70)+(0,-0.1cm)--node[midway,yshift=0.2cm, xshift=0.8cm] 
{\footnotesize $m$ }(a77)--++(0,0.5cm);
\end{tikzpicture}
\end{center}
       \item in $S_3$ exactly all simplices $\sigma$ marked in $\DeltaThree{\ell}{}{\sharp}\otimes \Delta[m]$ and not in $S_2$ with $\sigma''=[12]$ by means of thinness extensions.
        The generic simplex $\sigma$ that is being marked in $S_3$ can be depicted as follows.
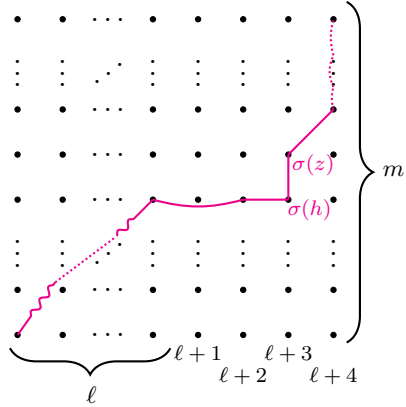
\begin{center}
\begin{tikzpicture}
\def\l{0.6cm}
\draw[fill] (0,0) circle (1pt) node (a00){};
\draw[fill] (\l, 0) circle (1pt) node (a10){};
\draw (2*\l, 0) node(a20){$\ldots$};
\draw[fill] (3*\l, 0) circle (1pt) node (a30){};
\draw[fill] (0,\l) circle (1pt) node (a01){};
\draw[fill] (\l, \l) circle (1pt) node (a11){};
\draw (2*\l, \l) node(a21){$\ldots$};
\draw[fill] (3*\l, \l) circle (1pt) node (a31){};
\draw (0,2*\l) node[inner sep=0, outer sep=0] (a02){$\vdots$};
\draw (\l,2*\l) node[inner sep=0, outer sep=0]  (a12){$\vdots$};
\draw (2*\l,2*\l) node[inner sep=0, outer sep=0]  (a22){$\iddots$};
\draw (3*\l,2*\l) node[inner sep=0, outer sep=0]  (a32){$\vdots$};
\draw[fill] (0,3*\l) circle (1pt) node (a03){};
\draw[fill] (\l, 3*\l) circle (1pt) node (a13){};
\draw (2*\l, 3*\l) node(a23){$\ldots$};
\draw[fill] (3*\l, 3*\l) circle (1pt) node (a33){};
\draw[fill] (4*\l, 0) circle (1pt) node (a40){} node[below](l+1){{\scriptsize $\ell+1$}};
\draw[fill] (5*\l, 0) circle (1pt) node (a50){} node[below, yshift=-0.3cm](l+2){{\scriptsize $\ell+2$}};
\draw[fill] (6*\l, 0) circle (1pt) node (a60){} node[below](l+3){{\scriptsize $\ell+3$}};
\draw[fill] (7*\l, 0) circle (1pt) node (a70){} node[below, yshift=-0.3cm](l+4){{\scriptsize $\ell+4$}};
\draw[fill] (4*\l, \l) circle (1pt) node (a41){};
\draw[fill] (5*\l, \l) circle (1pt) node (a51){};
\draw[fill] (6*\l, \l) circle (1pt) node (a61){};
\draw[fill] (7*\l, \l) circle (1pt) node (a71){};
\draw (4*\l,2*\l) node[inner sep=0, outer sep=0] (a42){$\vdots$};
\draw (5*\l,2*\l) node[inner sep=0, outer sep=0]  (a52){$\vdots$};
\draw (6*\l,2*\l) node[inner sep=0, outer sep=0]  (a62){$\vdots$};
\draw (7*\l,2*\l) node[inner sep=0, outer sep=0]  (a72){$\vdots$};
\draw[fill] (4*\l, 3*\l) circle (1pt) node (a43){};
\draw[fill] (5*\l, 3*\l) circle (1pt) node (a53){};
\draw[fill] (6*\l, 3*\l) circle (1pt) node (a63){};
\draw[fill] (7*\l, 3*\l) circle (1pt) node (a73){};
\draw[fill] (0,4*\l) circle (1pt) node (a04){};
\draw[fill] (\l, 4*\l) circle (1pt) node (a14){};
\draw (2*\l, 4*\l) node(a24){$\ldots$};
\draw[fill] (3*\l, 4*\l) circle (1pt) node (a34){};
\draw[fill] (4*\l, 4*\l) circle (1pt) node (a44){};
\draw[fill] (5*\l, 4*\l) circle (1pt) node (a54){};
\draw[fill] (6*\l, 4*\l) circle (1pt) node (a64){};
\draw[fill] (7*\l, 4*\l) circle (1pt) node (a74){};
\draw[fill] (0,5*\l) circle (1pt) node (a05){};
\draw[fill] (\l, 5*\l) circle (1pt) node (a15){};
\draw (2*\l, 5*\l) node(a25){$\ldots$};
\draw[fill] (3*\l, 5*\l) circle (1pt) node (a35){};
\draw[fill] (4*\l, 5*\l) circle (1pt) node (a45){};
\draw[fill] (5*\l, 5*\l) circle (1pt) node (a55){};
\draw[fill] (6*\l, 5*\l) circle (1pt) node (a65){};
\draw[fill] (7*\l, 5*\l) circle (1pt) node (a75){};
\draw (0,6*\l) node[inner sep=0, outer sep=0] (a06){$\vdots$};
\draw (\l,6*\l) node[inner sep=0, outer sep=0]  (a16){$\vdots$};
\draw (2*\l,6*\l) node[inner sep=0, outer sep=0]  (a26){$\iddots$};
\draw (3*\l,6*\l) node[inner sep=0, outer sep=0]  (a36){$\vdots$};
\draw (4*\l,6*\l) node[inner sep=0, outer sep=0] (a46){$\vdots$};
\draw (5*\l,6*\l) node[inner sep=0, outer sep=0]  (a56){$\vdots$};
\draw (6*\l,6*\l) node[inner sep=0, outer sep=0]  (a66){$\vdots$};
\draw (7*\l,6*\l) node[inner sep=0, outer sep=0]  (a76){$\vdots$};
\draw[fill] (0,7*\l) circle (1pt) node (a07){};
\draw[fill] (\l, 7*\l) circle (1pt) node (a17){};
\draw (2*\l, 7*\l) node(a27){$\ldots$};
\draw[fill] (3*\l, 7*\l) circle (1pt) node (a37){};
\draw[fill] (4*\l, 7*\l) circle (1pt) node (a47){};
\draw[fill] (5*\l, 7*\l) circle (1pt) node (a57){};
\draw[fill] (6*\l, 7*\l) circle (1pt) node (a67){};
\draw[fill] (7*\l, 7*\l) circle (1pt) node (a77){};
\draw[decorate,decoration={snake,amplitude=.4mm,segment length=2mm,pre length=0.3cm}, thick, magenta] (a00.center)--(0.5,0.7);
\draw (a33.center)+(-0.5,-0.5) node[inner sep=0, outer sep=0](targ){};
\draw[densely dotted, magenta, thick] (0.5, 0.7)--(targ);
\draw[decorate,decoration={snake,amplitude=.4mm,segment length=2mm,post length=0.3cm}, thick, magenta] (targ)--(a33.center);
\draw[thick, magenta] (a33.center) to [bend right=15](a53.center);
\draw[thick, magenta] (a53.center) to (a63.center);
\draw[thick, magenta] (a63.center) to (a64.center);
\draw[thick, magenta] (a64.center) to (a75.center);
\draw[thick, magenta, densely dotted, bend left=15] (a75.center) to ++(0,0.5*\l);
\draw[thick, magenta, densely dotted, bend left=25](a75.center) ++(0,0.5*\l) to ++(0,0.7*\l);
\draw[thick, magenta, densely dotted, bend left=15](a75.center) ++(0,1.2*\l)to (a77.center);
\draw[magenta] (a64.south east) node[xshift=0.2cm](labelz){{\scriptsize $\sigma(z)$}};
\draw[magenta] (a63.south east) node[xshift=0.15cm](labelh){{\scriptsize $\sigma(h)$}};
\draw[decorate, decoration= {brace, raise=5pt, amplitude=10pt, mirror}, thick] (a00)+(-0.1cm,0)--node[midway,yshift=-0.8cm, xshift=0.2cm]  {\footnotesize $\ell$ }(a30)--++(0.5cm,0);
\draw[decorate, decoration= {brace, raise=5pt, amplitude=10pt, mirror}, thick] (a70)+(0,-0.1cm)--node[midway,yshift=0.2cm, xshift=0.8cm] 
{\footnotesize $m$ }(a77)--++(0,0.5cm);
\end{tikzpicture}
\end{center}
        \item in $S_4$ exactly all simplices $\sigma$ marked in $\DeltaThree{\ell}{}{\sharp}\otimes \Delta[m]$ and not in $S_3$ with $\sigma''=[01]$ and $\pr_1\sigma$ hitting at most one of the values $l+3$ and $l+4$ by means of thinness extensions.
                The generic simplex $\sigma$ that is being marked in $S_4$ can be depicted as follows.
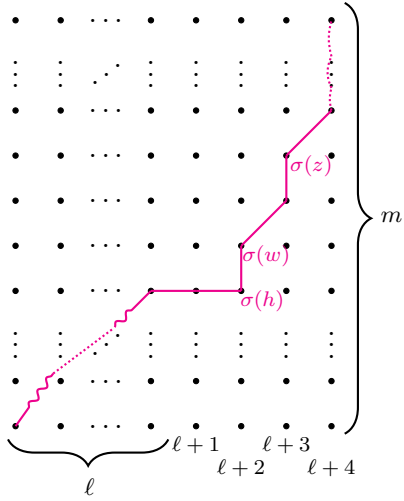
\begin{center}
\begin{tikzpicture}
\def\l{0.6cm}
\draw[fill] (0,0) circle (1pt) node (a00){};
\draw[fill] (\l, 0) circle (1pt) node (a10){};
\draw (2*\l, 0) node(a20){$\ldots$};
\draw[fill] (3*\l, 0) circle (1pt) node (a30){};
\draw[fill] (0,\l) circle (1pt) node (a01){};
\draw[fill] (\l, \l) circle (1pt) node (a11){};
\draw (2*\l, \l) node(a21){$\ldots$};
\draw[fill] (3*\l, \l) circle (1pt) node (a31){};
\draw (0,2*\l) node[inner sep=0, outer sep=0] (a02){$\vdots$};
\draw (\l,2*\l) node[inner sep=0, outer sep=0]  (a12){$\vdots$};
\draw (2*\l,2*\l) node[inner sep=0, outer sep=0]  (a22){$\iddots$};
\draw (3*\l,2*\l) node[inner sep=0, outer sep=0]  (a32){$\vdots$};
\draw[fill] (0,3*\l) circle (1pt) node (a03){};
\draw[fill] (\l, 3*\l) circle (1pt) node (a13){};
\draw (2*\l, 3*\l) node(a23){$\ldots$};
\draw[fill] (3*\l, 3*\l) circle (1pt) node (a33){};
\draw[fill] (4*\l, 0) circle (1pt) node (a40){} node[below](l+1){{\scriptsize $\ell+1$}};
\draw[fill] (5*\l, 0) circle (1pt) node (a50){} node[below, yshift=-0.3cm](l+2){{\scriptsize $\ell+2$}};
\draw[fill] (6*\l, 0) circle (1pt) node (a60){} node[below](l+3){{\scriptsize $\ell+3$}};
\draw[fill] (7*\l, 0) circle (1pt) node (a70){} node[below, yshift=-0.3cm](l+4){{\scriptsize $\ell+4$}};
\draw[fill] (4*\l, \l) circle (1pt) node (a41){};
\draw[fill] (5*\l, \l) circle (1pt) node (a51){};
\draw[fill] (6*\l, \l) circle (1pt) node (a61){};
\draw[fill] (7*\l, \l) circle (1pt) node (a71){};
\draw (4*\l,2*\l) node[inner sep=0, outer sep=0] (a42){$\vdots$};
\draw (5*\l,2*\l) node[inner sep=0, outer sep=0]  (a52){$\vdots$};
\draw (6*\l,2*\l) node[inner sep=0, outer sep=0]  (a62){$\vdots$};
\draw (7*\l,2*\l) node[inner sep=0, outer sep=0]  (a72){$\vdots$};
\draw[fill] (4*\l, 3*\l) circle (1pt) node (a43){};
\draw[fill] (5*\l, 3*\l) circle (1pt) node (a53){};
\draw[fill] (6*\l, 3*\l) circle (1pt) node (a63){};
\draw[fill] (7*\l, 3*\l) circle (1pt) node (a73){};
\draw[fill] (0,4*\l) circle (1pt) node (a04){};
\draw[fill] (\l, 4*\l) circle (1pt) node (a14){};
\draw (2*\l, 4*\l) node(a24){$\ldots$};
\draw[fill] (3*\l, 4*\l) circle (1pt) node (a34){};
\draw[fill] (4*\l, 4*\l) circle (1pt) node (a44){};
\draw[fill] (5*\l, 4*\l) circle (1pt) node (a54){};
\draw[fill] (6*\l, 4*\l) circle (1pt) node (a64){};
\draw[fill] (7*\l, 4*\l) circle (1pt) node (a74){};
\draw[fill] (0,5*\l) circle (1pt) node (a05){};
\draw[fill] (\l, 5*\l) circle (1pt) node (a15){};
\draw (2*\l, 5*\l) node(a25){$\ldots$};
\draw[fill] (3*\l, 5*\l) circle (1pt) node (a35){};
\draw[fill] (4*\l, 5*\l) circle (1pt) node (a45){};
\draw[fill] (5*\l, 5*\l) circle (1pt) node (a55){};
\draw[fill] (6*\l, 5*\l) circle (1pt) node (a65){};
\draw[fill] (7*\l, 5*\l) circle (1pt) node (a75){};
\draw (0,6*\l) node[inner sep=0, outer sep=0] (a06){$\vdots$};
\draw (\l,6*\l) node[inner sep=0, outer sep=0]  (a16){$\vdots$};
\draw (2*\l,6*\l) node[inner sep=0, outer sep=0]  (a26){$\iddots$};
\draw (3*\l,6*\l) node[inner sep=0, outer sep=0]  (a36){$\vdots$};
\draw (4*\l,6*\l) node[inner sep=0, outer sep=0] (a46){$\vdots$};
\draw (5*\l,6*\l) node[inner sep=0, outer sep=0]  (a56){$\vdots$};
\draw (6*\l,6*\l) node[inner sep=0, outer sep=0]  (a66){$\vdots$};
\draw (7*\l,6*\l) node[inner sep=0, outer sep=0]  (a76){$\vdots$};
\draw[fill] (0,7*\l) circle (1pt) node (a07){};
\draw[fill] (\l, 7*\l) circle (1pt) node (a17){};
\draw (2*\l, 7*\l) node(a27){$\ldots$};
\draw[fill] (3*\l, 7*\l) circle (1pt) node (a37){};
\draw[fill] (4*\l, 7*\l) circle (1pt) node (a47){};
\draw[fill] (5*\l, 7*\l) circle (1pt) node (a57){};
\draw[fill] (6*\l, 7*\l) circle (1pt) node (a67){};
\draw[fill] (7*\l, 7*\l) circle (1pt) node (a77){};
\draw[decorate,decoration={snake,amplitude=.4mm,segment length=2mm,pre length=0.3cm}, thick, magenta] (a00.center)--(0.5,0.7);
\draw (a33.center)+(-0.5,-0.5) node[inner sep=0, outer sep=0](targ){};
\draw[densely dotted, magenta, thick] (0.5, 0.7)--(targ);
\draw[decorate,decoration={snake,amplitude=.4mm,segment length=2mm,post length=0.3cm}, thick, magenta] (targ)--(a33.center);
\draw[thick, magenta] (a33.center) to (a43.center);
\draw[thick, magenta] (a43.center) to (a53.center);
\draw[thick, magenta] (a53.center) to (a54.center);
\draw[thick, magenta] (a54.center) to (a75.center);
\draw[thick, magenta, densely dotted, bend left=15] (a75.center) to ++(0,0.5*\l);
\draw[thick, magenta, densely dotted, bend left=25](a75.center) ++(0,0.5*\l) to ++(0,0.7*\l);
\draw[thick, magenta, densely dotted, bend left=15](a75.center) ++(0,1.2*\l)to (a77.center);
\draw[magenta] (a54.south east) node[xshift=0.2cm](labelz){{\scriptsize $\sigma(z)$}};
\draw[magenta] (a53.south east) node[xshift=0.15cm](labelh){{\scriptsize $\sigma(h)$}};
\draw[decorate, decoration= {brace, raise=5pt, amplitude=10pt, mirror}, thick] (a00)+(-0.1cm,0)--node[midway,yshift=-0.8cm, xshift=0.2cm]  {\footnotesize $\ell$ }(a30)--++(0.5cm,0);
\draw[decorate, decoration= {brace, raise=5pt, amplitude=10pt, mirror}, thick] (a70)+(0,-0.1cm)--node[midway,yshift=0.2cm, xshift=0.8cm] 
{\footnotesize $m$ }(a77)--++(0,0.5cm);
\end{tikzpicture}
\end{center}
    \item in $S_5$ exactly all simplices $\sigma$ marked in $\DeltaThree{\ell}{}{\sharp}\otimes \Delta[m]$ and not in $S_4$ with $\sigma''=[01]$ and $\pr_1\sigma$ hitting both $l+3$ and $l+4$, with last appearances of $l+2$ and $l+3$ in consecutive positions by means of thinness extensions. 
            The generic simplex $\sigma$ that is being marked in $S_5$ can be depicted as follows.
\begin{center}
\begin{tikzpicture}
\def\l{0.6cm}
\draw[fill] (0,0) circle (1pt) node (a00){};
\draw[fill] (\l, 0) circle (1pt) node (a10){};
\draw (2*\l, 0) node(a20){$\ldots$};
\draw[fill] (3*\l, 0) circle (1pt) node (a30){};
\draw[fill] (0,\l) circle (1pt) node (a01){};
\draw[fill] (\l, \l) circle (1pt) node (a11){};
\draw (2*\l, \l) node(a21){$\ldots$};
\draw[fill] (3*\l, \l) circle (1pt) node (a31){};
\draw (0,2*\l) node[inner sep=0, outer sep=0] (a02){$\vdots$};
\draw (\l,2*\l) node[inner sep=0, outer sep=0]  (a12){$\vdots$};
\draw (2*\l,2*\l) node[inner sep=0, outer sep=0]  (a22){$\iddots$};
\draw (3*\l,2*\l) node[inner sep=0, outer sep=0]  (a32){$\vdots$};
\draw[fill] (0,3*\l) circle (1pt) node (a03){};
\draw[fill] (\l, 3*\l) circle (1pt) node (a13){};
\draw (2*\l, 3*\l) node(a23){$\ldots$};
\draw[fill] (3*\l, 3*\l) circle (1pt) node (a33){};
\draw[fill] (4*\l, 0) circle (1pt) node (a40){} node[below](l+1){{\scriptsize $\ell+1$}};
\draw[fill] (5*\l, 0) circle (1pt) node (a50){} node[below, yshift=-0.3cm](l+2){{\scriptsize $\ell+2$}};
\draw[fill] (6*\l, 0) circle (1pt) node (a60){} node[below](l+3){{\scriptsize $\ell+3$}};
\draw[fill] (7*\l, 0) circle (1pt) node (a70){} node[below, yshift=-0.3cm](l+4){{\scriptsize $\ell+4$}};
\draw[fill] (4*\l, \l) circle (1pt) node (a41){};
\draw[fill] (5*\l, \l) circle (1pt) node (a51){};
\draw[fill] (6*\l, \l) circle (1pt) node (a61){};
\draw[fill] (7*\l, \l) circle (1pt) node (a71){};
\draw (4*\l,2*\l) node[inner sep=0, outer sep=0] (a42){$\vdots$};
\draw (5*\l,2*\l) node[inner sep=0, outer sep=0]  (a52){$\vdots$};
\draw (6*\l,2*\l) node[inner sep=0, outer sep=0]  (a62){$\vdots$};
\draw (7*\l,2*\l) node[inner sep=0, outer sep=0]  (a72){$\vdots$};
\draw[fill] (4*\l, 3*\l) circle (1pt) node (a43){};
\draw[fill] (5*\l, 3*\l) circle (1pt) node (a53){};
\draw[fill] (6*\l, 3*\l) circle (1pt) node (a63){};
\draw[fill] (7*\l, 3*\l) circle (1pt) node (a73){};
\draw[fill] (0,4*\l) circle (1pt) node (a04){};
\draw[fill] (\l, 4*\l) circle (1pt) node (a14){};
\draw (2*\l, 4*\l) node(a24){$\ldots$};
\draw[fill] (3*\l, 4*\l) circle (1pt) node (a34){};
\draw[fill] (4*\l, 4*\l) circle (1pt) node (a44){};
\draw[fill] (5*\l, 4*\l) circle (1pt) node (a54){};
\draw[fill] (6*\l, 4*\l) circle (1pt) node (a64){};
\draw[fill] (7*\l, 4*\l) circle (1pt) node (a74){};
\draw[fill] (0,5*\l) circle (1pt) node (a05){};
\draw[fill] (\l, 5*\l) circle (1pt) node (a15){};
\draw (2*\l, 5*\l) node(a25){$\ldots$};
\draw[fill] (3*\l, 5*\l) circle (1pt) node (a35){};
\draw[fill] (4*\l, 5*\l) circle (1pt) node (a45){};
\draw[fill] (5*\l, 5*\l) circle (1pt) node (a55){};
\draw[fill] (6*\l, 5*\l) circle (1pt) node (a65){};
\draw[fill] (7*\l, 5*\l) circle (1pt) node (a75){};
\begin{scope}[yshift=\l]
\draw (0,6*\l) node[inner sep=0, outer sep=0] (a06){$\vdots$};
\draw (\l,6*\l) node[inner sep=0, outer sep=0]  (a16){$\vdots$};
\draw (2*\l,6*\l) node[inner sep=0, outer sep=0]  (a26){$\iddots$};
\draw (3*\l,6*\l) node[inner sep=0, outer sep=0]  (a36){$\vdots$};
\draw (4*\l,6*\l) node[inner sep=0, outer sep=0] (a46){$\vdots$};
\draw (5*\l,6*\l) node[inner sep=0, outer sep=0]  (a56){$\vdots$};
\draw (6*\l,6*\l) node[inner sep=0, outer sep=0]  (a66){$\vdots$};
\draw (7*\l,6*\l) node[inner sep=0, outer sep=0]  (a76){$\vdots$};
\draw[fill] (0,7*\l) circle (1pt) node (a07){};
\draw[fill] (\l, 7*\l) circle (1pt) node (a17){};
\draw (2*\l, 7*\l) node(a27){$\ldots$};
\draw[fill] (3*\l, 7*\l) circle (1pt) node (a37){};
\draw[fill] (4*\l, 7*\l) circle (1pt) node (a47){};
\draw[fill] (5*\l, 7*\l) circle (1pt) node (a57){};
\draw[fill] (6*\l, 7*\l) circle (1pt) node (a67){};
\draw[fill] (7*\l, 7*\l) circle (1pt) node (a77){};
\end{scope}
\draw[fill] (0,6*\l) circle (1pt) node (b06){};
\draw[fill] (\l, 6*\l) circle (1pt) node (b16){};
\draw (2*\l, 6*\l) node(b26){$\ldots$};
\draw[fill] (3*\l, 6*\l) circle (1pt) node (b36){};
\draw[fill] (4*\l, 6*\l) circle (1pt) node (b46){};
\draw[fill] (5*\l, 6*\l) circle (1pt) node (b56){};
\draw[fill] (6*\l, 6*\l) circle (1pt) node (b66){};
\draw[fill] (7*\l, 6*\l) circle (1pt) node (b76){};
\draw[decorate,decoration={snake,amplitude=.4mm,segment length=2mm,pre length=0.3cm}, thick, magenta] (a00.center)--(0.5,0.7);
\draw (a33.center)+(-0.5,-0.5) node[inner sep=0, outer sep=0](targ){};
\draw[densely dotted, magenta, thick] (0.5, 0.7)--(targ);
\draw[decorate,decoration={snake,amplitude=.4mm,segment length=2mm,post length=0.3cm}, thick, magenta] (targ)--(a33.center);
\draw[thick, magenta] (a33.center) to (a43.center);
\draw[thick, magenta] (a43.center) to (a53.center);
\draw[thick, magenta] (a53.center) to (a54.center);
\draw[thick, magenta] (a54.center) to (a65.center);
\draw[thick, magenta] (a65.center) to (b76.center);
\draw[thick, magenta, densely dotted, bend left=15] (b76.center) to ++(0,0.5*\l);
\draw[thick, magenta, densely dotted, bend left=25](b76.center) ++(0,0.5*\l) to ++(0,0.7*\l);
\draw[thick, magenta, densely dotted, bend left=15](b76.center) ++(0,1.2*\l)to (a77.center);
\draw[magenta] (a65.south east) node[xshift=0.2cm](labelz){{\scriptsize $\sigma(z)$}};
\draw[magenta] (a53.south east) node[xshift=0.15cm](labelh){{\scriptsize $\sigma(h)$}};
\draw[decorate, decoration= {brace, raise=5pt, amplitude=10pt, mirror}, thick] (a00)+(-0.1cm,0)--node[midway,yshift=-0.8cm, xshift=0.2cm]  {\footnotesize $\ell$ }(a30)--++(0.5cm,0);
\draw[decorate, decoration= {brace, raise=5pt, amplitude=10pt, mirror}, thick] (a70)+(0,-0.1cm)--node[midway,yshift=0.2cm, xshift=0.8cm] 
{\footnotesize $m$ }(a77)--++(0,0.5cm);
\end{tikzpicture}
\end{center}
        \item in $S_6$ exactly all simplices $\sigma$ marked in $\DeltaThree{\ell}{}{\sharp}\otimes \Delta[m]$ and not in $S_5$ (which in particular have $\sigma''=[01]$ and $\pr_1\sigma$ hitting both $\ell+3$ and $\ell+4$, with last appearances of $\ell+2$ and $\ell+3$ not in consecutive positions) by means of thinness extensions.
                The generic simplex $\sigma$ that is being marked in $S_3$ can be depicted as follows.
\begin{center}
\begin{tikzpicture}
\def\l{0.6cm}
\draw[fill] (0,0) circle (1pt) node (a00){};
\draw[fill] (\l, 0) circle (1pt) node (a10){};
\draw (2*\l, 0) node(a20){$\ldots$};
\draw[fill] (3*\l, 0) circle (1pt) node (a30){};
\draw[fill] (0,\l) circle (1pt) node (a01){};
\draw[fill] (\l, \l) circle (1pt) node (a11){};
\draw (2*\l, \l) node(a21){$\ldots$};
\draw[fill] (3*\l, \l) circle (1pt) node (a31){};
\draw (0,2*\l) node[inner sep=0, outer sep=0] (a02){$\vdots$};
\draw (\l,2*\l) node[inner sep=0, outer sep=0]  (a12){$\vdots$};
\draw (2*\l,2*\l) node[inner sep=0, outer sep=0]  (a22){$\iddots$};
\draw (3*\l,2*\l) node[inner sep=0, outer sep=0]  (a32){$\vdots$};
\draw[fill] (0,3*\l) circle (1pt) node (a03){};
\draw[fill] (\l, 3*\l) circle (1pt) node (a13){};
\draw (2*\l, 3*\l) node(a23){$\ldots$};
\draw[fill] (3*\l, 3*\l) circle (1pt) node (a33){};
\draw[fill] (4*\l, 0) circle (1pt) node (a40){} node[below](l+1){{\scriptsize $\ell+1$}};
\draw[fill] (5*\l, 0) circle (1pt) node (a50){} node[below, yshift=-0.3cm](l+2){{\scriptsize $\ell+2$}};
\draw[fill] (6*\l, 0) circle (1pt) node (a60){} node[below](l+3){{\scriptsize $\ell+3$}};
\draw[fill] (7*\l, 0) circle (1pt) node (a70){} node[below, yshift=-0.3cm](l+4){{\scriptsize $\ell+4$}};
\draw[fill] (4*\l, \l) circle (1pt) node (a41){};
\draw[fill] (5*\l, \l) circle (1pt) node (a51){};
\draw[fill] (6*\l, \l) circle (1pt) node (a61){};
\draw[fill] (7*\l, \l) circle (1pt) node (a71){};
\draw (4*\l,2*\l) node[inner sep=0, outer sep=0] (a42){$\vdots$};
\draw (5*\l,2*\l) node[inner sep=0, outer sep=0]  (a52){$\vdots$};
\draw (6*\l,2*\l) node[inner sep=0, outer sep=0]  (a62){$\vdots$};
\draw (7*\l,2*\l) node[inner sep=0, outer sep=0]  (a72){$\vdots$};
\draw[fill] (4*\l, 3*\l) circle (1pt) node (a43){};
\draw[fill] (5*\l, 3*\l) circle (1pt) node (a53){};
\draw[fill] (6*\l, 3*\l) circle (1pt) node (a63){};
\draw[fill] (7*\l, 3*\l) circle (1pt) node (a73){};
\draw[fill] (0,4*\l) circle (1pt) node (a04){};
\draw[fill] (\l, 4*\l) circle (1pt) node (a14){};
\draw (2*\l, 4*\l) node(a24){$\ldots$};
\draw[fill] (3*\l, 4*\l) circle (1pt) node (a34){};
\draw[fill] (4*\l, 4*\l) circle (1pt) node (a44){};
\draw[fill] (5*\l, 4*\l) circle (1pt) node (a54){};
\draw[fill] (6*\l, 4*\l) circle (1pt) node (a64){};
\draw[fill] (7*\l, 4*\l) circle (1pt) node (a74){};
\draw[fill] (0,5*\l) circle (1pt) node (a05){};
\draw[fill] (\l, 5*\l) circle (1pt) node (a15){};
\draw (2*\l, 5*\l) node(a25){$\ldots$};
\draw[fill] (3*\l, 5*\l) circle (1pt) node (a35){};
\draw[fill] (4*\l, 5*\l) circle (1pt) node (a45){};
\draw[fill] (5*\l, 5*\l) circle (1pt) node (a55){};
\draw[fill] (6*\l, 5*\l) circle (1pt) node (a65){};
\draw[fill] (7*\l, 5*\l) circle (1pt) node (a75){};
\begin{scope}[yshift=2*\l]
\draw (0,6*\l) node[inner sep=0, outer sep=0] (a06){$\vdots$};
\draw (\l,6*\l) node[inner sep=0, outer sep=0]  (a16){$\vdots$};
\draw (2*\l,6*\l) node[inner sep=0, outer sep=0]  (a26){$\iddots$};
\draw (3*\l,6*\l) node[inner sep=0, outer sep=0]  (a36){$\vdots$};
\draw (4*\l,6*\l) node[inner sep=0, outer sep=0] (a46){$\vdots$};
\draw (5*\l,6*\l) node[inner sep=0, outer sep=0]  (a56){$\vdots$};
\draw (6*\l,6*\l) node[inner sep=0, outer sep=0]  (a66){$\vdots$};
\draw (7*\l,6*\l) node[inner sep=0, outer sep=0]  (a76){$\vdots$};
\draw[fill] (0,7*\l) circle (1pt) node (a07){};
\draw[fill] (\l, 7*\l) circle (1pt) node (a17){};
\draw (2*\l, 7*\l) node(a27){$\ldots$};
\draw[fill] (3*\l, 7*\l) circle (1pt) node (a37){};
\draw[fill] (4*\l, 7*\l) circle (1pt) node (a47){};
\draw[fill] (5*\l, 7*\l) circle (1pt) node (a57){};
\draw[fill] (6*\l, 7*\l) circle (1pt) node (a67){};
\draw[fill] (7*\l, 7*\l) circle (1pt) node (a77){};
\end{scope}
\draw[fill] (0,6*\l) circle (1pt) node (b06){};
\draw[fill] (\l, 6*\l) circle (1pt) node (b16){};
\draw (2*\l, 6*\l) node(b26){$\ldots$};
\draw[fill] (3*\l, 6*\l) circle (1pt) node (b36){};
\draw[fill] (4*\l, 6*\l) circle (1pt) node (b46){};
\draw[fill] (5*\l, 6*\l) circle (1pt) node (b56){};
\draw[fill] (6*\l, 6*\l) circle (1pt) node (b66){};
\draw[fill] (7*\l, 6*\l) circle (1pt) node (b76){};
\draw[fill] (0,7*\l) circle (1pt) node (b07){};
\draw[fill] (\l, 7*\l) circle (1pt) node (b17){};
\draw (2*\l, 7*\l) node(b27){$\ldots$};
\draw[fill] (3*\l, 7*\l) circle (1pt) node (b37){};
\draw[fill] (4*\l, 7*\l) circle (1pt) node (b47){};
\draw[fill] (5*\l, 7*\l) circle (1pt) node (b57){};
\draw[fill] (6*\l, 7*\l) circle (1pt) node (b67){};
\draw[fill] (7*\l, 7*\l) circle (1pt) node (b77){};
\draw[decorate,decoration={snake,amplitude=.4mm,segment length=2mm,pre length=0.3cm}, thick, magenta] (a00.center)--(0.5,0.7);
\draw (a33.center)+(-0.5,-0.5) node[inner sep=0, outer sep=0](targ){};
\draw[densely dotted, magenta, thick] (0.5, 0.7)--(targ);
\draw[decorate,decoration={snake,amplitude=.4mm,segment length=2mm,post length=0.3cm}, thick, magenta] (targ)--(a33.center);
\draw[thick, magenta] (a33.center) to (a43.center);
\draw[thick, magenta] (a43.center) to (a53.center);
\draw[thick, magenta] (a53.center) to (a54.center);
\draw[thick, magenta] (a54.center) to (a65.center);
\draw[thick, magenta] (a65.center) to (b66.center);
\draw[thick, magenta] (b66.center) to (b77.center);
\draw[thick, magenta, densely dotted, bend left=15] (b77.center) to ++(0,0.5*\l);
\draw[thick, magenta, densely dotted, bend left=25](b77.center) ++(0,0.5*\l) to ++(0,0.7*\l);
\draw[thick, magenta, densely dotted, bend left=15](b77.center) ++(0,1.2*\l)to (a77.center);
\draw[magenta] (b66.south east) node[xshift=0.2cm](labelz){{\scriptsize $\sigma(z)$}};
\draw[magenta] (a54.south east) node[xshift=0.2cm](labelz){{\scriptsize $\sigma(w)$}};
\draw[magenta] (a53.south east) node[xshift=0.15cm](labelh){{\scriptsize $\sigma(h)$}};
\draw[decorate, decoration= {brace, raise=5pt, amplitude=10pt, mirror}, thick] (a00)+(-0.1cm,0)--node[midway,yshift=-0.8cm, xshift=0.2cm]  {\footnotesize $\ell$ }(a30)--++(0.5cm,0);
\draw[decorate, decoration= {brace, raise=5pt, amplitude=10pt, mirror}, thick] (a70)+(0,-0.1cm)--node[midway,yshift=0.2cm, xshift=0.8cm] 
{\footnotesize $m$ }(a77)--++(0,0.5cm);
\end{tikzpicture}
\end{center}
        \end{enumerate}
        We now proceed to explaining how to build the desired filtrations.
\begin{enumerate}[leftmargin=*]
    \item For $\ell=-1$ we set $S_1=S_0$, and for $\ell>-1$ we will obtain $S_1$ from $S_0$ by marking exactly all simplices $\sigma$ marked in $\DeltaThree{\ell}{}{\sharp}\otimes \Delta[m]$ and not in $S_0$ that are contained in a copy of $\DeltaThree{\ell-1}{}{\sharp}\otimes \Delta[m]\hookrightarrow\DeltaThree{\ell}{}{\sharp}\otimes \Delta[m]$.
    For each $0\leq i\leq \ell$, we consider the map of simplicial sets with marking $\DeltaThree{\ell-1}{}{eq}\otimes \Delta[m]\to S_0$ induced by the $i$-th face,
    and we can then express the inclusions $S_0\hookrightarrow S_1$ as the pushout with a disjoint union of the inclusion
    \[\DeltaThree{\ell-1}{}{\sharp}\otimes\partial \Delta[m]   \aamalg{\DeltaThree{\ell-1}{}{eq} \otimes \partial \Delta[m]} \DeltaThree{\ell-1}{}{eq}\otimes \Delta[m] \hookrightarrow\DeltaThree{\ell-1}{}{\sharp}\otimes \Delta[m]\]
    which are acyclic cofibrations given by the induction hypothesis:
   \[
\begin{tikzcd}[column sep=small]
\coprod\limits_{i\in [l]} \DeltaThree{\ell-1}{}{\sharp}\otimes\partial \Delta[m]   \aamalg{\DeltaThree{\ell-1}{}{eq} \otimes \partial \Delta[m]} \DeltaThree{\ell-1}{}{eq}\otimes \Delta[m] \arrow[r]\arrow[d]& \coprod\limits_{i\in [l]} \DeltaThree{\ell-1}{}{\sharp}\otimes \Delta[m]\arrow[d]\\
S_0\arrow[r]& S_1.
\end{tikzcd}
\]
In particular, $S_0\hookrightarrow S_1$ is an acyclic cofibration.
Moreover, we have an induced inclusion
\[S_1\hookrightarrow\DeltaThree{\ell}{}{\sharp}\otimes \Delta[m].\]
\item We obtain $S_2$ from $S_1$ by marking in particular for $m \leq r \leq \ell+4+m$ all $r$-simplices $\sigma$ marked in $\DeltaThree{\ell}{}{\sharp}\otimes \Delta[m]$ and not in $S_1$ with $\sigma''\in\{[03], [23]\}$, as well as some additional simplices. 
For any $m\leq r\leq \ell+4+m$ and any
\[\vec b:=\left(b_{0}\leq \ldots \leq b_{\ell}\right)\]
an increasing sequence in $[0,\ell+4+m-r]$,
we argue that the simplicial map
\[\varphi\colon\Delta[l]\star \Delta[3]\star \Delta[r-\ell-5]\to \Delta[\ell+4] \times \Delta[m]\] defined by the formula 
\[i\mapsto
\left\{\begin{array}{ll}
(i, b_i) & \mbox{ if }0\leq i \leq {\ell}\\
    (i,\ell+4+m-r) & \mbox{ if }\ell+1\leq i \leq \ell+4\\
     (\ell+4, m-r+i)&  \mbox{ if }\ell+5\leq i\leq r\\
\end{array}
\right.
\]
is in particular a map of simplicial sets with marking $\DeltaThree{\ell}{r-\ell-5}{eq}\to S_1$.

To see this, we suppose that
\[\gamma_1\star \gamma_2\star \gamma_3\colon \Delta[r_1]\star \Delta[r_2]\star \Delta[r_3]\to \Delta[\ell]\star \Delta[3]\star \Delta[r-\ell-5]\]
is a generic marked and non-degenerate $(r_1+1+r_2+1+r_3)$-simplex of $\Delta[\ell]\star \Delta[3]\star \Delta[r-\ell-5]$ and we prove that the $(r_1+1+r_2+1+r_3)$-simplex of $\Delta[\ell+4]\times\Delta[m]$ defined by the composite of maps of simplicial sets
\[\Delta[r_1]\star \Delta[r_2]\star \Delta[r_3]\xrightarrow{\gamma_1\star\gamma_2\star\gamma_3} \Delta[\ell]\star \Delta[3]\star \Delta[r-\ell-5]\xrightarrow{\varphi} \Delta[\ell+4]\times\Delta[m]\] 
is marked in $\DeltaThree{\ell}{}{eq}\otimes \Delta[m]$. Since $\gamma_1\star\gamma_2\star\gamma_3$ is marked, one amongst the $\gamma_i$'s must be marked, and since moreover $\gamma_1\star\gamma_2\star\gamma_3$ is non-degenerate the simplex $\gamma_2$ must be marked in $\Delta[3]_{eq}$. By construction, the degeneracy index of the composite $\varphi\circ(\gamma_1\star\gamma_2\star\gamma_3)$
is $r_3+1$. Moreover, we see that the partition face $\amalg_1^{r_1+1+r_2, r_3+1}$ of the first component of $\varphi\circ(\gamma_1\star\gamma_2\star\gamma_3)$ is of the form
\[
\Delta[r_1]\star \Delta[r_2] \xrightarrow{\gamma_1\star \gamma_2} \Delta[\ell]\star \Delta[3]\to 
\DeltaThree{\ell}{}{eq}
\]
and it is marked because it is the join of the marked simplex $\gamma_2\colon\Delta[r_2]\to\Delta[3]_{eq}$ with another simplex of the form $\Delta[r_1]\to\Delta[\ell]$. This proves that the simplicial map $\varphi$ does indeed preserve the marking.

We then define the inclusion $S_1\hookrightarrow S_2$ as the pushout with the union of a family of saturation extensions (which are acyclic cofibrations by \cref{generalsaturation}) of the form $\DeltaThree{\ell}{r-\ell-5}{eq}\to\DeltaThree{\ell}{r-\ell-5}{\sharp}$:
\[
\begin{tikzcd}
\coprod\limits_{r}\coprod\limits_{\vec b}\DeltaThree{\ell}{r-\ell-5}{eq} \arrow[r]\arrow[d]& \coprod\limits_{r}\coprod\limits_{\vec b}\DeltaThree{\ell}{r-\ell-5}{\sharp}\arrow[d]\\
S_1 \arrow[r]& S_2.
\end{tikzcd}
\]
In particular, $S_1\hookrightarrow S_2$ is an acyclic cofibration and we have added all simplices $\sigma$ marked in $\DeltaThree{\ell}{}{\sharp}\otimes \Delta[m]$ and not in $S_1$ with $\sigma''\in\{[03], [23]\}$.
Moreover, with a reasoning similar to the one producing the map $\DeltaThree{\ell}{r-\ell-5}{eq}\to S_1$, one can show that there is an induced map
\[S_2\hookrightarrow\DeltaThree{\ell}{}{\sharp}\otimes \Delta[m].\]
    \item We obtain $S_3$ from $S_2$ by marking for $m \leq r \leq \ell+4+m$ all $r$-simplices $\sigma$ marked in $\DeltaThree{\ell}{}{\sharp}\otimes \Delta[m]$ and not in $S_2$ with $\sigma''=[12]$.
    For any such $\sigma$ in question there is a degeneracy index $h>1$ and a unique maximal $h\leq z\leq r$ so that $\pr_1\sigma(z)=\ell+3$. In particular, $r-m\leq z\leq r$.
    The new markings will be added by constructing a sequence of acyclic cofibrations
    \[S_2=:S_2^{(0)}\hookrightarrow S_2^{(1)}\hookrightarrow\dots\hookrightarrow S_2^{(z)}\hookrightarrow S_2^{(z+1)}\hookrightarrow\dots\hookrightarrow S_2^{(\ell+4+m)}=:S_3\]
    such that $S_2^{(z)}$ contains all missing markings for simplices of a given $z$.
For any $\sigma$ with a given $z$, the simplicial map \[\psi\colon\Delta[r+1]\to\Delta[\ell+4]\times \Delta[m]\]
defined by the formula
\[i\mapsto
\left\{\begin{array}{ll}
    \sigma(i) & \mbox{ if }0\leq i\leq z \\
     (\ell+4, m-r+z)&  \mbox{ if }i=z+1\\
          \sigma(i-1)&  \mbox{ if }z+1<i\leq r+1
\end{array}
\right.
\]
is in particular a map of simplicial sets with marking $\Delta^{z+1}[r+1]'\to S_2^{(z-1)}$.

To see this, we consider a non-degenerate marked $s$-simplex $\tau\colon\Delta[s]\to \Delta[r+1]$ of $\Delta^{z+1}[r+1]'$, and we prove that the $s$-simplex of $\Delta[\ell+4]\times\Delta[m]$ defined by the composite of map of simplicial sets
\[\psi\circ\tau\colon \Delta[s]\xrightarrow{\tau}\Delta[r+1]\xrightarrow{\psi} \Delta[\ell+4]\times \Delta[m]\] is marked in $S_2^{(z-1)}$.
\begin{itemize}[leftmargin=*]
    \item If $\tau$ contains $\{z, z+1, z+2\}\cap[r+1]$, by construction the second projection of $\psi\circ\tau$ is degenerate, with degeneracy index being the preimage of $z+1$ in $\Delta[s]$. Moreover, the face partition of the first component of $\psi\circ\tau$ 
    contains
    the edge $[(\ell+3)(\ell+4)]$ in $\Delta[\ell+4]$ and so the simplex $\psi\circ \tau$ is marked in $S_2$. 
    \item If $\tau=d^{z+2}$, by construction the second projection of $\psi\circ\tau$ is not surjective (as it misses the value $m-r+z+1$) and moreover degenerate, with degeneracy index being the preimage of $z+1$ in $\Delta[r]$. Moreover, the face partition of the first component of $\psi\circ\tau$
    hits at least a $1$-dimensional simplex of $\Delta[3]$. In particular, $\psi\circ\tau$ is marked already in $\DeltaThree{\ell}{}{\sharp}\otimes\partial \Delta[m]$.
    \item If $\tau=d^{z}$, we distinguish two cases. If $z=h$, by construction the second projection of $\psi\circ\tau$ is degenerate, with degeneracy index $z=h$. Moreover, the face partition of the first component of $\psi\circ\tau$ contains the edge $[(\ell+2)(\ell+4)]$ in $\Delta[\ell+4]$ and so the simplex $\psi\circ \tau$ is marked in $S_2$. If $h<z$, by construction the second projection of $\psi\circ\tau$ is degenerate, with degeneracy index $h$. Moreover, the face partition of the first component of $\psi\circ\tau$ contains the edge $[(\ell+2)(\ell+3)]$ in $\Delta[\ell+4]$ and in fact the marking of the simplex $\psi\circ \tau$ was added in $S_2^{(z-1)}$.
\end{itemize}
This proves that the simplicial map $\psi$ does indeed preserve the marking.

 We then define the inclusion $S_2^{(z-1)}\hookrightarrow S_2^{(z)}$ as the pushout with several thinness extensions $\Delta^{z+1}[r+1]'\to\Delta^{z+1}[r+1]''$ (as many as $r$-simplices $\sigma$ as $z$ varies): 
\[
\begin{tikzcd}
\coprod\limits_{r}\coprod\limits_{z}\coprod\limits_{\sigma}\Delta^{z+1}[r+1]' \arrow[r]\arrow[d]& \coprod\limits_{r}\coprod\limits_{z}\coprod\limits_{\sigma}\Delta^{z+1}[r+1]''\arrow[d]\\
S_2^{(z-1)} \arrow[r]& S_2^{(z)}.
\end{tikzcd}
\]
In particular $S_2^{(z-1)}\hookrightarrow S_2^{(z)}$ is an acyclic cofibration. 
Moreover, by construction there is an induced map
\[S_2^{(z)}\hookrightarrow \DeltaThree{\ell}{}{\sharp}\otimes \Delta[m].\]
 We then set $S_3:= S_2^{(r)}$, so that in particular $S_2\hookrightarrow S_3$ is an acyclic cofibration and we have marked all simplices $\sigma$ marked in $\DeltaThree{\ell}{}{\sharp}\otimes \Delta[m]$ and not in $S_2$ with $\sigma''=[12]$.
Moreover, by construction we have an induced map
\[S_3\hookrightarrow \DeltaThree{\ell}{}{\sharp}\otimes \Delta[m]\]

    \item We obtain $S_4$ from $S_3$ by marking for $m \leq r \leq \ell+4+m$ all $r$-simplices $\sigma$ marked in $\DeltaThree{\ell}{}{\sharp}\otimes \Delta[m]$ and not in $S_3$ with $\sigma''=[01]$ and $\pr_1\sigma$ hitting at most one of the values $\ell+3$ and $\ell+4$.
    For any such $\sigma$ in question there is a unique maximal $h\leq z\leq r$ so that $\pr_1\sigma(z)=\ell+2$. In particular, $r-m\leq z\leq r$. We will add the missing simplices by constructing a sequence of anodyne extensions
    \[S_3=:S_3^{(0)}\hookrightarrow S_3^{(1)}\hookrightarrow\dots\hookrightarrow S_3^{(z-1)}\hookrightarrow S_3^{(z)}\hookrightarrow\dots \hookrightarrow S_3^{(\ell+4+m)}=:S_4\]
    such that $S_3^{(z)}$ contains all missing simplices for a given $z$.
For any $\sigma$ with a given $z$, the simplicial map \[\psi\colon\Delta[r+1]\to\Delta[\ell+4]\times \Delta[m]\]
defined by the formula
\[i\mapsto
\left\{\begin{array}{ll}
    \sigma(i) & \mbox{ if }0\leq i\leq z \\
     (\ell+4, z)&  \mbox{ if }i=z+1\mbox{ and }\pr_1\sigma(z+1)=\ell+4 \mbox{ or }z=r,\\
    (\ell+3, z)&  \mbox{ if }i=z+1\mbox{ and }\pr_1\sigma(z+1)=\ell+3,\\
          \sigma(i-1)&  \mbox{ if }z+1<i\leq r+1
\end{array}
\right.
\]
is in particular a map of simplicial sets with marking $\Delta^{z+1}[r+1]'\to S_3^{(z-1)}$.

To see this, we consider a non-degenerate marked $s$-simplex $\tau\colon\Delta[s]\to \Delta[r+1]$ of $\Delta^{z+1}[r+1]'$, and we prove that the $s$-simplex of $\Delta[\ell+4]\times\Delta[m]$ defined by the composite of maps of simplicial sets
\[\psi\circ\tau\colon \Delta[s]\xrightarrow{\tau}\Delta[r+1]\xrightarrow{\psi} \Delta[\ell+4]\times \Delta[m]\] is marked in $S_3^{(z-1)}$.
\begin{itemize}[leftmargin=*]
    \item If $\tau$ contains $\{z,z+1,z+2\}\cap[r+1]$, by construction the second projection of $\psi\circ\tau$ is degenerate, with degeneracy index being the preimage of $z+1$ in $\Delta[s]$. Moreover, the face partition of the first component of $\psi\circ\tau$ 
    contains
    the edge $[(\ell+2)(\ell+3)]$ or $[(\ell+2)(\ell+4)]$ in $\Delta[\ell+4]$ and so the simplex $\psi\circ \tau$ is marked in $S_3$. 
    \item If $\tau=d^{z+2}$, by construction the second projection of $\psi\circ\tau$ is not surjective (as it misses the value $m-r+z+1$) and moreover degenerate, with degeneracy index being the preimage of $z+1$ in $\Delta[r]$. Moreover, the face partition of the first component of $\psi\circ\tau$
    hits at least a $1$-dimensional simplex of $\Delta[3]$. In particular, $\psi\circ\tau$ is marked already in $\DeltaThree{\ell}{}{\sharp}\otimes\partial \Delta[m]$.
    \item If $\tau=d^{z}$, we distinguish two cases. If $h=z$, by construction the second projection of $\psi\circ\tau$ is degenerate, with degeneracy index $h=z$. Moreover, the face partition of the first component of $\psi\circ\tau$ contains the edge $[(\ell+1)(\ell+3)]$ or $[(\ell+1)(\ell+4)]$ in $\Delta[\ell+4]$ and so the simplex $\psi\circ \tau$ is marked in $S_2$. If $h<z$, by construction the second projection of $\psi\circ\tau$ is degenerate, with degeneracy index $h$. Moreover, the face partition of the first component of $\psi\circ\tau$ contains the edge $[(\ell+1)(\ell+2)]$ in $\Delta[\ell+4]$ and in fact the marking of the simplex $\psi\circ \tau$ was added in $S_3^{(z-1)}$.
\end{itemize}
This proves that the simplicial map $\psi$ does indeed preserve the marking.

 We then define the inclusion $S_3^{(z-1)}\hookrightarrow S_3^{(z)}$ as the pushout with many thinness anodyne extensions $\Delta^{z+1}[r+1]'\to\Delta^{z+1}[r+1]''$ (as many as $r$-simplices $\sigma$ as $z$ varies):
\[
\begin{tikzcd}
\coprod\limits_{r}\coprod\limits_{z}\coprod\limits_{\sigma}\Delta^{z+1}[r+1]' \arrow[r]\arrow[d]& \coprod\limits_{r}\coprod\limits_{z}\coprod\limits_{\sigma}\Delta^{z+1}[r+1]''\arrow[d]\\
S_3^{(z-1)} \arrow[r]& S_3^{(z)}.
\end{tikzcd}
\]
In particular $S_3^{(z-1)}\hookrightarrow S_3^{(z)}$ is an acyclic cofibration.
Moreover, we have an induced map
\[S_4^{(z)}\hookrightarrow\DeltaThree{\ell}{}{\sharp}\otimes \Delta[m]\]
 We then set
 $S_4:=S_3^{(\ell+4+m)}$, so that in particular $S_3\hookrightarrow S_4$ is an acyclic cofibration and we have marked all simplices $\sigma$ marked in $\DeltaThree{\ell}{}{\sharp}\otimes \Delta[m]$ and not in $S_3$ with $\sigma''=[01]$ and $\pr_1\sigma$ hitting at most one of the values $\ell+3$ and $\ell+4$.
Moreover, we have an induced map
\[S_4\hookrightarrow\DeltaThree{\ell}{}{\sharp}\otimes \Delta[m].\]

\item We obtain $S_5$ from $S_4$ by marking for $m \leq r \leq \ell+4+m$ all $r$-simplices $\sigma$ marked in $\DeltaThree{\ell}{}{\sharp}\otimes \Delta[m]$ and not in $S_4$ with $\sigma''=[01]$ and $\pr_1\sigma$ hitting both $\ell+3$ and $\ell+4$, with last appearances of $\ell+2$ and $\ell+3$ in consecutive positions. 
    More precisely, for any such $\sigma$ in question there is a unique maximal $h+1\leq z\leq r-1$ so that $\pr_1\sigma(z)=\ell+3$, and by assumption $\pr_1\sigma(z+1)=\ell+4$ and $\pr_2\sigma(z-1)=\ell+2$. In particular, $r-m+1\leq z\leq \ell+3+m$.
For any $\sigma$ with a given $z$, the simplicial map
\[\psi\colon\Delta[r+1]\to\Delta[\ell+4]\times \Delta[m]\]
defined by the formula
\[i\mapsto
\left\{\begin{array}{ll}
    \sigma(i) & \mbox{ if }0\leq i\leq z \\
     (m-r+z, \ell+4)&  \mbox{ if }i=z+1\\
          \sigma(i-1)&  \mbox{ if }z+1<i\leq r+1
\end{array}
\right.
\]
is in particular a map of simplicial sets with marking $\Delta^{z+1}[r+1]'\to S_4$.

To see this, we consider a non-degenerate marked $s$-simplex $\tau\colon\Delta[s]\to \Delta[r+1]$ of $\Delta^{z+1}[r+1]'$, and we prove that the $s$-simplex of $\Delta[\ell+4]\times\Delta[m]$ defined by the composite of maps of simplicial sets
\[\psi\circ\tau\colon \Delta[s]\xrightarrow{\tau}\Delta[r+1]\xrightarrow{\psi} \Delta[\ell+4]\times \Delta[m]\] is marked in $S_4$.
\begin{itemize}[leftmargin=*]
    \item If $\tau$ contains $\{z,z+1,z+2\}\cap[r+1]$, by construction the second projection of $\psi\circ\tau$ is degenerate, with degeneracy index being the preimage of $z+1$ in $\Delta[s]$. Moreover, the face partition of the first component of $\psi\circ\tau$ 
    contains
    the edge $[(\ell+3)(\ell+4)]$ in $\Delta[\ell+4]$ and so the simplex $\psi\circ \tau$ is marked in $S_2$. 
    \item If $\tau=d^{z+2}$, by construction the second projection of $\psi\circ\tau$ is not surjective (as it misses the value $m-r+z+1$) and moreover degenerate, with degeneracy index being the preimage of $z+1$ in $\Delta[r]$. Moreover, the face partition of the second component of $\psi\circ\tau$ 
    hits at least a $1$-dimensional simplex of $\Delta[3]$. In particular, $\psi\circ\tau$ is marked already in $\DeltaThree{\ell}{}{\sharp}\otimes\partial \Delta[m]$.
    \item If $\tau=d^{z}$, by construction the second projection of $\psi\circ\tau$ is degenerate, with degeneracy index $h$. Moreover, the face partition of the first component of $\psi\circ\tau$ contains the edge $[(\ell+1)(\ell+2)]$ in $\Delta[\ell+4]$ and does not hit $\ell+3$, so the marking of the simplex $\psi\circ \tau$ was added in $S_4$.
\end{itemize}
This proves that the simplicial map $\psi$ does indeed preserve the marking.

 We define the inclusion $S_4\hookrightarrow S_5$ as the pushout with several thinness extensions $\Delta^{z+1}[r+1]'\to\Delta^{z+1}[r+1]''$ (as many as $r$-simplices $\sigma$ as $z$ varies):
\[
\begin{tikzcd}
\coprod\limits_{r}\coprod\limits_{z}\coprod\limits_{\sigma}\Delta^{z+1}[r+1]' \arrow[r]\arrow[d]& \coprod\limits_{r}\coprod\limits_{z}\coprod\limits_{\sigma}\Delta^{z+1}[r+1]''\arrow[d]\\
S_4 \arrow[r]& S_5.
\end{tikzcd}
\]
In particular $S_4\hookrightarrow S_5$ is an acyclic cofibration and we have marked all simplices in $\DeltaThree{\ell}{}{\sharp}\otimes \Delta[m]$ and not in $S_4$ with $\sigma''=[01]$ and $\pr_1\sigma$ hitting both the values $\ell+3$ and $\ell+4$, with last appearances of $\ell+2$ and $\ell+3$ in consecutive positions.
Moreover, we have an induced map
\[S_5\hookrightarrow\DeltaThree{\ell}{}{\sharp}\otimes \Delta[m].\]
 \item We obtain $S_6$ from $S_5$ by marking for $m \leq r \leq \ell+4+m$ all missing $r$-simplices with $\sigma''=[01]$ and $\pr_1\sigma$ hitting both $\ell+3$ and $\ell+4$, with last appearances of $\ell+2$ and $\ell+3$ not in consecutive positions. 
 More precisely, for any such $\sigma$ in question there is a unique maximal $h< z<r$ so that $\pr_1\sigma(z)=\ell+3$.
 In particular, $r-m\leq z\leq \ell+3+m$. We will add them by constructing a sequence of acyclic cofibrations
    \[S_5=:S_5^{(r-m)}\hookrightarrow S_5^{(r-m+1)}\hookrightarrow\dots\hookrightarrow S_5^{(z-1)}\hookrightarrow S_5^{(z)}\hookrightarrow\dots\hookrightarrow S_5^{(\ell+3+m)}=:S_6\]
    such that $S_5^{(z)}$ contains all missing simplices for a given $z$.
For any $\sigma$ with a given $z$, the simplicial map
\[\psi\colon\Delta[r+1]\to\Delta[\ell+4]\times \Delta[m]\]
defined by the formula
\[i\mapsto
\left\{\begin{array}{ll}
    \sigma(i) & \mbox{ if }0\leq i\leq z \\
     (\ell+4, m-r+z)&  \mbox{ if }i=z+1\\
          \sigma(i-1)&  \mbox{ if }z+1<i\leq r+1
\end{array}
\right.
\]
is in particular a map of simplicial sets with marking $\Delta^{z+1}[r+1]'\to S_5^{(z-1)}$.

To see this, we consider a non-degenerate marked $s$-simplex $\tau\colon\Delta[s]\to \Delta[r+1]$ of $\Delta^{z+1}[r+1]'$, and we prove that the $s$-simplex of $\Delta[\ell+4]\times\Delta[m]$ defined by the composite of maps of simplicial sets
\[\psi\circ\tau\colon \Delta[s]\xrightarrow{\tau}\Delta[r+1]\xrightarrow{\psi} \Delta[\ell+4]\times \Delta[m]\] is marked in $S_5^{(z-1)}$.
\begin{itemize}[leftmargin=*]
    \item If $\tau$ contains $\{z, z+1, z+2\}\cap[r+1]$, by construction the second projection of $\psi\circ\tau$ is degenerate, with degeneracy index being the preimage of $z+1$ in $\Delta[s]$. Moreover, the face partition of the first component of $\psi\circ\tau$ 
    contains 
    the edge $[(\ell+3)(\ell+4)]$ in $\Delta[\ell+4]$ and so the simplex $\psi\circ \tau$ is marked in $S_2$. 
    \item If $\tau=d^{z+2}$, by construction the second projection of $\psi\circ\tau$ is not surjective (as it misses the value $m-r+z+1$) and moreover degenerate, with degeneracy index being the preimage of $z+1$ in $\Delta[r]$. Moreover, the face partition of the first component of $\psi\circ\tau$
    hits at least a $1$-dimensional simplex of $\Delta[3]$. In particular, $\psi\circ\tau$ is marked already in $\DeltaThree{\ell}{}{\sharp}\otimes\partial \Delta[m]$.
    \item If $\tau=d^{z}$, we distinguish two cases depending on the value of $w$, being the maximal value for which $\pr_1\sigma(w)=\ell+2$. By assumption, $h\leq w <z-1$. If $w=z-2$, by construction the second projection of $\psi\circ\tau$ is degenerate, with degeneracy index $h$. Moreover, the face partition of the first component of $\psi\circ\tau$ 
    contains the edge $[(\ell+1)(\ell+2)]$ in $\Delta[\ell+4]$ and hits $\ell+2$ and $\ell+3$ in consecutive positions for the last time and so the marking of $\psi\circ \tau$ was added in $S_5$. If $w<z-2$, by construction the second projection of $\psi\circ\tau$ is degenerate, with degeneracy index $h$. Moreover, the face partition of the second component of $\psi\circ\tau$ 
    contains the edge $[(\ell+1)(\ell+2)]$ in $\Delta[\ell+4]$ and in fact the marking of the simplex $\psi\circ \tau$ was added in $S_5^{(z-1)}$.
\end{itemize}
This proves that the simplicial map $\psi$ does indeed preserve the marking.

 We then define the inclusion $S_5^{(z-1)}\hookrightarrow S_5^{(z)}$ as the pushout with several thinness extensions $\Delta^{z+1}[r+1]'\to\Delta^{z+1}[r+1]''$ (as many as $r$-simplices $\sigma$ as $z$ varies):
\[
\begin{tikzcd}
\coprod\limits_{r}\coprod\limits_{z}\coprod\limits_{\sigma}\Delta^{z+1}[r+1]' \arrow[r]\arrow[d]& \coprod\limits_{r}\coprod\limits_{z}\coprod\limits_{\sigma}\Delta^{z+1}[r+1]''\arrow[d]\\
S_5^{(z-1)} \arrow[r]& S_5^{(z)}.
\end{tikzcd}
\]
In particular $S_5^{(z-1)}\hookrightarrow S_5^{(z)}$ is an acyclic cofibration.
Moreover, we have an induced map
\[S_5^{(z)}\hookrightarrow\DeltaThree{\ell}{}{\sharp}\otimes \Delta[m].\]
 We then set %
 $S_6:=S_5^{(\ell+3+m)}$, so that in particular $S_5\hookrightarrow S_6$ is an acyclic cofibration and we have marked all simplices $\sigma$ marked in $\DeltaThree{\ell}{}{\sharp}\otimes \Delta[m]$ and not in $S_5$ with $\sigma''=[01]$ and $\pr_1\sigma$ hitting both the values $\ell+3$ and $\ell+4$, with last appearances of $l+2$ and $l+3$ not in consecutive positions.
In particular, we have an isomorphism
\[S_6\cong\DeltaThree{\ell}{}{\sharp}\otimes \Delta[m].\]
\end{enumerate}
This concludes the proof.
\end{proof}

\bibliographystyle{amsalpha}
\bibliography{ref}
\end{document}